\documentclass[preprint]{elsarticle}

\usepackage{amssymb,amsmath}
\usepackage{amsthm}
\usepackage{ifthen}
\usepackage{tikz}
\usepackage{framed}
\usepackage{framed}

\newtheorem{theorem}{Theorem}[section]
\newif\ifskip
\skiptrue
\newif\ifrevised
\revisedtrue

\newtheorem{remark}[theorem]{\bf Remark}
\newtheorem{example}[theorem]{\bf Example}
\newtheorem{lemma}[theorem]{\bf Lemma}
\newtheorem{definition}[theorem]{\bf Definition}
\newtheorem{proposition}[theorem]{\bf Proposition}
\newtheorem{corollary}[theorem]{\bf Corollary}

\newtheorem{Definitions}[theorem]{\bf Definitions}

\newtheorem{Examples}[theorem]{\bf Examples}

\newtheorem{problem}[theorem]{\bf Problem}
\skiptrue

\newcommand{\N}{{\mathbb N}}

\newcommand{\Z}{{\mathbb Z}} 
\newcommand{\R}{{\mathbb R}} 
\newcommand{\C}{{\mathbb C}} 
\newcommand{\Q}{{\mathbb Q}} 
\newcommand{\bP}{{\mathbf P}} 
\newcommand{\bQ}{{\mathbf Q}}

\newcommand{\cH}{\mathcal{H}} 
\newcommand{\cC}{\mathcal{C}} 
\newcommand{\cR}{\mathcal{R}} 
\newcommand{\cD}{\mathcal{D}} 
\newcommand{\cS}{\mathcal{S}} 
\newcommand{\cG}{\mathcal{G}} 
\newcommand{\cK}{\mathcal{K}} 
\newcommand{\bF}{\mathbf{F}} 
\newcommand{\bX}{\mathbf{X}} 
\newcommand{\bY}{\mathbf{Y}} 
 
\newcommand{\ba}{\mathbf{a}} 
\newcommand{\bb}{\mathbf{b}} 
\newcommand{\bB}{\mathbf{B}} 
\newcommand{\bv}{\mathbf{v}} 
\newcommand{\bi}{\mathbf{\i}} 
\newcommand{\bU}{\mathbf{U}} 
\newcommand{\per}{\mathrm{per}} 
\newcommand{\fA}{\mathfrak{A}} 

\newcommand{\SOL}{\mathrm{SOL}}
\newcommand{\MSOL}{\mathrm{MSOL}}
\newcommand{\SOLEVAL}{\mathrm{SOLEVAL}}

\begin{document}
\begin{frontmatter}
\title{A Logician's View of Graph Polynomials
}

\author[jam]{J.A.~Makowsky\corref{cor1}\fnref{fn1}}
\ead{janos@cs.technion.ac.il}
\ead[url]{http://www.cs.technion.ac.il/~janos}

\author[evr]{E.V.~Ravve\fnref{fn2}}
\ead{cselena@braude.ac.il}

\author[tk]{T.~Kotek\fnref{fn0}}
\ead{kotek@forsyte.at}
\ead[url]{http://forsyte.at/people/kotek}

\cortext[cor1]{Corresponding author}
\fntext[fn0]{Work done in part while the author was visiting the Simons Institute
for the Theory of Computing in Fall 2016.}
\fntext[fn1]{Work done in part while the author was visiting the Simons Institute
for the Theory of Computing in Spring and Fall 2016.}
\fntext[fn2]{Visiting scholar at the Faculty of Computer Science, Technion--IIT, Haifa, Israel}
\address[tk]{
Institut f\"ur Informationssysteme, Technische Universit\"at Wien Vienna, Austria}
\address[jam]{Department of Computer Science, Technion--IIT, Haifa, Israel}
\address[evr]{Department of Software Engineering, ORT-Braude College, Karmiel, Israel}
\begin{abstract}
Graph polynomials are graph parameters invariant under graph isomorphisms
which take values in a polynomial ring with a fixed finite number of indeterminates.
We study graph polynomials from a model theoretic point of view.
In this paper we distinguish between the graph theoretic (semantic) and the algebraic (syntactic)
meaning of graph polynomials. 
Graph polynomials appear in the literature either as generating functions,
as generalized chromatic polynomials, or as polynomials derived via determinants of
adjacency or Laplacian matrices.
We show that these forms are mutually incomparable, and propose
a unified framework based on definability in Second Order Logic.
We show that this comprises virtually all examples of graph polynomials with a fixed finite
set of indeterminates. 
Finally we show that the location of zeros and stability of graph polynomials 
is not a semantic property.
The paper emphasizes a model theoretic view. It gives a unified exposition of classical
results in algebraic combinatorics together with new and
some of our previously obtained results scattered in the graph theoretic literature.
\end{abstract}

\end{frontmatter}
\newpage
\small
\tableofcontents
\newpage
\sloppy
\normalsize
\section{Introduction}
This paper gives a logician's view of some aspects of graph polynomials.
A short version was given as an invited lecture by the first author 
at WOLLIC 2016, \cite{pr:MakowskyRavve2016wollic}.

A graph $G=(V(G), E(G))$ is given by the set of vertices $V(G)$ and a symmetric edge-relation $E(G)$.
We denote by $n(G)$ the number of vertices, by $m(G)$ the number of edges, by
$k(G)$ the number of connected components of a graph $G$, and by $\mathcal{G}$
the class of finite graphs. 

Graph polynomials are graph invariants with values in a polynomial ring $\mathcal{R}$, 
usually $\Z[ \mathbf{X}]$ with 
$\mathbf{X}=(X_1, \ldots , X_{\ell})$.
Let $\bP(G;\mathbf{X})$ be a graph polynomial of the form
$$
\bP(G;\mathbf{X}) = \sum_{i_1, \ldots, i_{\ell}=0}^{d(G)} c_{i_1, \ldots, i_{\ell}}(G) 
X_1^{i_1}\cdot \ldots \cdot X_{\ell}^{i_{\ell}},
$$
where $\mathbf{X}=(X_1, \ldots, X_{\ell})$,
$d(G)$ is a graph parameter with non-negative integers as its values,
and 
$$c_{i_1, \ldots, i_{\ell}}(G): i_1, \ldots, i_{\ell} \leq d(G)$$
are integer valued graph parameters.

\begin{definition}
\label{def:computable}
A graph polynomial
{\em $\bP(G; \bX)$ is computable} if 
\begin{enumerate}[(i)]
\item
$\bP(G; \bX)$ is a Turing computable function,
and additionally, 
\item
the range of $\bP(G; \bX)$, the set
$$
\{ p(\bX) \in \Z[\bX] : \mbox{ there is a graph  } G \mbox{  with  } \bP(G;\bX) = p(\bX) \}
$$
is Turing decidable.
\end{enumerate}
\end{definition}
The second condition is needed to make Theorem \ref{th:equiv} work.

Graph polynomials have been studied for the last hundred years, since G. Birkhoff introduced
his chromatic polynomial in \cite{ar:Birkhoff1912}. 
This was generalized by H. Whitney in the 1930ties, \cite{ar:Whitney32} and led to the 
Tutte polynomial, also called the {\em dichromate} or the {\em Tutte-Whitney polynomial}.
For a history see \cite{ar:Farr2007}.
Motivated by questions in theoretical
chemistry, the characteristic polynomial and the matching polynomial of graphs were introduced,
and studied intensively,
\cite{ar:Hosoya1971,ar:HeilmannLieb72,bk:Trinajstic1992,ar:Balaban93,ar:Balaban95,ar:Hosoya2002}.
In the last 30 years many more graph polynomials appeared in the literature.
The abundance of graph polynomials which appear in the more recent literature leads to various questions:
\begin{itemize}
\item 
How to compare graph polynomials?
\item 
What kind of information may be extracted from a graph polynomial about its
underlying graph?
\item
Are there any normal forms of graph polynomials?
\end{itemize}

Ten years ago B. Zilber and the first author have discovered a connection between model theory
and graph polynomials, \cite{ar:MakowskyZilber2006,ar:KotekMakowskyZilber11}.
In \cite{ar:MakowskyRavveBlanchard2014} we introduced the distinction between syntactic and semantic properties
of graph polynomials. 
In logic two formulas are semantically (logically) equivalent if they have the same models,
in other words, if they do not distinguish between two models.
Syntactic properties of formulas refer to properties of the string which is the formula.
Prenex normal form is a syntactic property. Semantic properties of a formula are properties of the class of models
of this formula shared by the class of models of logically equivalent formulas.
For graph polynomials $P(G;\bX)$ syntactic properties are properties of the particular polynomials
$P(G;\bX)$ for each $G$, whereas two graph polynomials are semantically equivalent if they do not distinguish
between any pair of graphs, or graphs with the same number of vertices, edges and connected components.
Our discussion in the above and subsequent papers,
was mostly addressed the graph theory community.
This paper is written for the logically minded and 
is a continuation of our analysis of notions used in the literature on graph polynomials.
\subsection{Why study graph polynomials?}
The first graph polynomial, the chromatic polynomial, was introduced in 1912 by G. Birkhoff 
to study the Four Color Conjecture, \cite{ar:Birkhoff1912}. 
The emergence of the Tutte polynomial can be seen as an attempt to generalize the chromatic polynomial,
cf. \cite{ar:Tutte54,bk:Bollobas98,ar:Ellis-MonaghanMerino2011}.
The characteristic polynomial and the matching polynomial were introduced with applications 
from chemistry in mind, cf.
\cite{bk:Trinajstic1992,ar:Balaban93,ar:Balaban95,bk:CvetkovicDoobSachs1995,bk:BrouwerHaemers2012}.
Physicists study various partition functions in statistical mechanics, 
in percolation theory and in the study of phase transitions, cf. \cite{bk:NesetrilWinkler04}.
It turns out that many partition functions are incarnations of the Tutte polynomial.
Another incarnation of the Tutte polynomial  is the Jones polynomial in Knot Theory, \cite{ar:Jaeger88} 
and again \cite{bk:Bollobas98}.
The various incarnations of the Tutte polynomial have triggered an interest in other graph polynomials.
These graph polynomials are studied for various reasons:
\begin{itemize}
\item
Graph polynomials can be used to distinguish non-isomorphic graphs.
A graph polynomial is {\em complete} if it distinguishes all non-isomorphic graphs.
The quest for a complete graph polynomial which is also easy to compute failed so far for two reasons.
Either there were too many non-isomorphic graphs which could not be distinguished, and/or
the proposed graph polynomial was more difficult to compute than just checking graph isomorphism.
\item
New graph polynomials may appear when we model 
behavior of physical, chemical or biological systems.
The arguments whether a  graph polynomial is interesting,
depends on its success in predicting 
the behavior of the modeled systems.
Also the particular choice of the representation is dictated by the modeling process.
The fact that
the modeled process gives, in this case, rise to a particular graph polynomial, is secondary,
and the properties of the graph polynomial reflect more properties of the physical or chemical process
modeled, than properties of the underlying graph.
\item
New graph polynomials are also studied as part of graph theory proper.
Here one is interested in the interrelationship between various graph parameters without
particular applications in mind.
A graph polynomial is considered interesting {\em from a graph theoretic point of view}, 
if many graph parameters can be  (easily) derived from it.
\item
Graph polynomials are sometimes studied as a way of generating families of polynomials, irrespective
of their graph theoretic meaning. H. Wilf, \cite{ar:Wilf1973}
asked the question how to characterize the polynomials which
do occur as instances of chromatic polynomials of graphs as a family of polynomials.
We have addressed this approach to graph polynomials in \cite{ar:KotekMakowskyRavve2017arxiv}.
\end{itemize}
This paper deals only with the graph theoretic and logical aspects 
of graph polynomials, discarding the graph isomorphism problem
and  discarding the modeling of systems describing phenomena in the natural sciences.
We ultimately ask the question: When is a newly introduced graph polynomial
interesting from a graph theoretic or logical point of view and deserves to be studied, 
and what aspects are more rewarding in this study
than others. In particular, in the last part of this paper, 
we scrutinize the role of the location of the roots of specific graph polynomials
in terms of other graph theoretic properties.

\subsection{Why not just sequences of integers rather than polynomials?}

A graph polynomial $P(G;\bX)$ is uniquely determined by the 
sequences of its coefficients.
In practice these coefficients were usually chosen in a uniform way
having some combinatorial interpretation.
The function which associates such a sequence  with a polynomial
is rather artificial. Whether the coefficient $c_i$ is associated with $X^i$
or some other monomial $X^j$ or polynomial $g_i(X)$, or even a function
$f_i(X)$ which is not a polynomial in $X$ will depend on the graph theoretic 
question one wants to study. Historically, for the last hundred years, polynomials
were used. After it was discovered that the Four-Color-Conjecture can be formulated
as a problem about the chromatic polynomial, one is easily tempted to look for other
graph theoretic statements which be formulated in a similar way.
The abundance of such statements might make one believe that graph polynomials
are a good choice for studying graph theoretical questions.

It is conceivable that other ways of studying the same sequences are equally interesting,
but the fact is, that such investigations are absent from the literature.
Maybe our analysis of the way graph polynomials arise in general will
spur new lines of research, replacing graph polynomials by other algebraic or analytic formalisms.

\subsection{Why $\SOL$-definability?}
There are uncountably many graph polynomials if they are merely defined as graph invariants with values in
a polynomial ring.
We can impose more restrictions by imposing {\em computability} and, in an even more restrictive way, 
{\em definability} requirements.
Imposing complexity theoretic restrictions poses some serious problems, and
is studied in \cite{pr:MakowskyKotekRavve2013}.
However, it is not the subject of this paper. 

The earliest graph polynomials are the chromatic polynomial introduced in 1912, 
its generalization the Tutte polynomial, introduced in 1954 and the characteristic and matching
polynomials introduced in the 1950ties. It is not right away obvious how to find a common generalization.
The chromatic polynomial is a special case of a Harary polynomial (a generalization of the chromatic polynomial
introduced by Harary \cite{ar:GutmanHarary83}, the matching polynomial
is a special case of a generation function, and the characteristic polynomial is based on computing
a determinant of some matrix associated with a graph.

In \cite{ar:MakowskyTARSKI,ar:AverbouchGodlinMakowsky10,ar:KotekMakowskyZilber11,ar:GodlinKatzMakowsky12} 
the class of graph polynomials definable in Second Order Logic
($\SOL$),
is studied, which requires that $d(G)$ and 
$c_{\bi}(G)=c(G;\bi)$, $\bi =(i_1, \ldots, i_{\ell})$, are, even uniformly, definable in $\SOL$. 
With very few exceptions, the
graph polynomials studied in the literature are $\SOL$-definable\footnote{
Many are even definable in Monadic Second Order Logic $\MSOL$, \cite{ar:MakowskyZoo}.
The exceptions are in \cite{ar:NobleWelsh99}.
The algorithmic advantages of $\MSOL$-definability, \cite{ar:CourcelleMakowskyRoticsDAM} are
of no importance in this paper.
}.

It turns out that the $\SOL$-definable graph polynomials are the smallest class of graph polynomials
subject to some very natural closure properties which cover all the examples studied in the literature.
On the other hand we show in Sections 
\ref{se:many}
that certain naturally defined graph polynomials (the domination polynomial and certain
generalzed chromatic polynomials) cannot be written as generating fuctions or Harary polynomials.

We assume the reader is familiar with Finite Model Theory, cf. \cite{bk:EF95,bk:Libkin2004}.
The finite model theory of graph polynomials was developed in
\cite{ar:MakowskyZoo,phd:Kotek,ar:KotekMakowskyZilber11}.
For the convenience of the reader it will summarized in Section \ref{se:sol}.

Requiring that the graph polynomials are $\SOL$-definable also guarantees that their coefficients are
the result of counting {\em combinatorially meaningful} $\SOL$-definable configurations
in the underlying graph.

\subsection{On the location of roots of graph polynomials}
Up to this point we were mostly concerned with the logical presentation of graph polynomials
and justified, why our formalism of $\SOL$-definable graph polynomial is an appropriate choice.
A topic frequently studied in paper about graph polynomials is the location of the roots (zeroes)
of $P(G;\bX)$ for a fixed graph $G$.

Given a univariate graph polynomial $\bP(G;X)$
a complex number $z \in \C$ is a root of $\bP$ if there is a graph
$G$ such that $z$ is a root of $\bP(G;X)$.
Many results in the literature on graph polynomials deal with the location of
its roots. For multivariate graph polynomials the corresponding question is formulated in
terms of half-plane properties.
The location of the zeroes is a good question
to illustrate the difference between graph theoretic (semantic) and  algebraic (syntactic) properties
of graph polynomials.  
The last part of this paper shows that the location of roots is not a semantic property.

\ifskip\else
In this paper we justify, from a foundational point of view,
the definitions we have introduced in 
\cite{ar:MakowskyRavveBlanchard2014}. 
This concerns the 
restrictions of graph polynomials to graph polynomials definable in 
$\SOL$,
the various notions of equivalence of graph polynomials, 
the notions of syntactic and semantic properties of graph polynomials. 
\fi 

We first paraphrase
the main results of
\cite{ar:MakowskyRavveBlanchard2014}. 
These results are all of the form:
\begin{quote}
	(*)
	Let $U$ be a subset of the complex numbers, such as the reals, an open disk, 
	the lower or upper halfplane, or the complement thereof.
	Given a  univariate $\SOL$-definable graph polynomial $\bP(G;X)$,
	there exists a semantically equivalent $\SOL$-definable 
	graph polynomial $\bQ(G;X)$ with all its roots in $U$.
\end{quote}
They show, in a precise sense,
that the location of the roots of a  univariate graph polynomial is {\em not a semantic
property}. 
They are more of a {\em normal form property}:
Every univariate $\SOL$-definable graph polynomial $\bP(G;X)$ can be put into a semantically equivalent
form with prescribed location of its roots.

The proofs in
\cite{ar:MakowskyRavveBlanchard2014} have two parts: Finding $\bQ(G;X)$, and showing that this $\bQ(G;X)$
is $\SOL$-definable. Finding $\bQ(G;X)$ often uses some ``dirty trick'' from analysis, 
whereas showing $\SOL$-definability,
only sketched in
\cite{ar:MakowskyRavveBlanchard2014},
needs more efforts in the details.

In this paper
we extend results of 
\cite{ar:MakowskyRavveBlanchard2014} to multivariate graph polynomials $\bP(G;\bX)$.
We show that various versions of the "halfplane property" in higher dimensions
of multivariate graph polynomials are also not semantic properties of the underlying graph in the sense of (*). 
This is interesting for two reasons:
First, these halfplane properties were studied in the recent literature on graph polynomials, and,
second, the proofs that the constructed $\bQ(G;\bX)$ is $\SOL$-definable is much more complex.
For the convenience of the logically minded reader we repeat
many examples already discussed in 
\cite{ar:MakowskyRavveBlanchard2014}. Furthermore, 
we  provide in this paper the  details in proving $\SOL$-definability 
for the more difficult case of multivariate graph polynomials
and the various halfplane properties.

\subsection{Outline of the paper}
In Section \ref{se:csl-compare} we discuss the foundational aspects of comparing graph polynomials.
In Section \ref{apal-represent} we discuss different ways of representing graph polynomials
and introduce the notion of semantic (graph theoretic) 
and syntactic (algebraic) properties of graph polynomials.
In Section \ref{se:dp} we present the discussion of 
various notions of equivalence of graph polynomials based on their distinctive power
which also is part of Section \ref{se:many}.
In Section \ref{se:sol} we develop the framework of $\SOL$-definable graph polynomials.
This summarizes the framework given in \cite{phd:Kotek}.
In Section \ref{se:zeros} we discuss the location of zeros of graph polynomials. 
First, in Subsection \ref{csl-roots}, we review our previous results
previously published in \cite{ar:MakowskyRavveBlanchard2014}, 
which show that the
location of roots of univariate graph polynomials is not a semantic property.
Then, in Subsection \ref{se:stable}, 
we look at the multivariate version of the location of roots, 
the various halfplane properties, also called stability properties, and prove that
stability is also not a semantic property of multivariate graph polynomials.
The discussion of stable polynomials is appears for the first time in this paper.
Finally, in Section \ref{se:conclu} we draw our conclusions and formulate several open problems.

\ifskip
\else
In Section \ref{se:csl-compare} we discuss the foundational aspects of comparing graph polynomials,
and introduce the notion of semantic properties of graph polynomials.
In Section \ref{se:csl-roots} we discuss properties of univariate graph polynomials related to the
location or multiplicity of their roots.
In Section \ref{se:main}  
we look at the multivariate version of the location of roots, 
the various halfplane properties, also called stability properties, and prove that
stability is also not a semantic property of multivariate graph polynomials.
In Section \ref{se:conclu} we draw our conclusions and formulate several open problems.
\fi

\section{How to compare graph polynomials?}
\label{se:csl-compare}
Once the graph theorists started to study several graph polynomials, 
the need of comparing them naturally arises.
We analyze two notions of equivalence which both occur implicitely in the literature in many papers.
Authors will argue that their graph polynomial is different from other graph polynomials.
They will argue that
\begin{itemize}
\item
Some other graph polynomial is a special case of the newly studied graph polynomial.
\item
The newly introduced graph polynomial is incomparable to previously studied graph polynomials.
\item
The newly introduced graph polynomial is the most general graph polynomial
withing a certain class of graph invariants.
\end{itemize}
By analyzing the literature we extracted two ways of comparison, d.p.- and s.d.p.-equivalence,
which encompass all other notions used in various papers.
This section discusses the basic properties of these notions.
From a model-theoretic point of view, d.p.-equivalence is the more natural notion.
However, most graph theoretic papers compare the behaviour of graph polynomials  only
on graphs with the same number of vertices, edges and connected
components, which is captured by s.d.p.-equivalence.

For $\cR \in \{ \R, \C, \Z\}$
we denote by $\mathfrak{GP}_{\mathcal{R},r}$ the set of graph polynomials in $r$ indeterminates 
with coefficients in $\mathcal{R}$, and let $\bX= (X_1, \ldots, X_r)$ be $r$ indeterminates.
Let $\bP(G)=\bP(G;\bX)$ and $\bQ(G)=\bQ(G;\bX)$ be two graph polynomials.

The following statements appear frequently in the literature with the {\em intended meaning}, but  
{\em without a general definition}:
\begin{enumerate}
\item
$\bQ(G)$ is {\em a substitution instance} of $\bP(G)$.
\item
$\bQ(G)$ and $\bP(G)$ are really the same, up to a {\em prefactor}.
For example the various versions of the Tutte polynomial are said to be the same {\em up to a prefactor},
\cite{ar:Sokal2005a}, and the same holds for the various versions of the matching 
polynomial, \cite{bk:LovaszPlummer86}.
\item
$\bQ(G)$ is {\em at least as expressive} than $\bP(G)$.
\item
The coefficients of $\bP(G)$ can be determined, or even computed, from the coefficients of $\bQ(G)$.
\end{enumerate}
Usually these statements are understood to be uniform in the graphs $G$,
but this uniformity can take various forms.
In \cite{ar:MakowskyRavveBlanchard2014} we have given these statements precise meanings,
and we have initiated the analysis of their relationship.
In this paper we elaborate our approach from \cite{ar:MakowskyRavveBlanchard2014} further
with the logic community in mind.

From a model theoretic point of view a graph property is a Boolean graph parameter.
A closed formula in a logical formalism, say a fragment of 
$\SOL$,
is a syntactic object. Its meaning is given by a graph property, i.e., a class of finite graphs
closed under isomorphism. Two formulas are considered {\em logically equivalent} if they define the
same property. In other words, two formulas are considered equivalent if they do not distinguish
between two graphs. As Boolean graph parameters have only two possible values,
two formulas are equivalent if,  considered as graph parameters, they define the same function.

Let $\cR$ be a possibly infinite ring. An $\cR$-valued graph parameter is a function  $f$
which maps a graph $G$ into an element $f(G) \in \cR$.
A graph polynomial $\bP(G)$ is a graph parameter which takes values in a polynomial ring.

Graph parameters are {\em coextensive} if they define the same function.
However, co-extensiveness seems to be too strong a property to compare graph parameters.
For instance defining the size of a graph $G$ by its order $n(G) = |V(G)|$,
or by  $n'(G)= 2\cdot |V(G)|$, gives two non-coextensive graph parameters which still have the
same information content in the following sense.
For two $\cR$-valued graph parameters $f$ and $g$, we say that $g$ is at least as distinctive
as $f$, if  for two graphs $G_1, G_2$ $g$ does not distinguish between $G_1$ and $G_2$, i.e.,
$g(G_1)=g(G_2)$, then also $f$ does not distinguish between $G_1$ and $G_2$, i.e., $f(G_1) = f(G_2)$.

Graph theorists often compare the distinctive power of graph parameters on graphs which are not trivially
distinguishable. Here trivially distinguishable refers to different order, size or number of components.

\subsection{Equivalence of graph polynomials}
Let $\bP(G)$ be a graph polynomial.
We say that
two graphs $G,H$ are {\em similar} if they have the same number of vertices, edges and connected components.
A graph parameter or a graph polynomial is a {\em similarity function} 
if it is invariant under graph similarity.

Two graphs $G,H$ are {\em $\bP$-equivalent} if $\bP(H;\bX) = \bP(G;\bX)$.
{\em $\bP$  distinguishes between $G$ and $H$} if $G$ and $H$ are not $\bP$-equivalent.
Two graph polynomials $\bP(G;\bX)$ and $\bQ(G;\bY)$ with $r$ and $s$ indeterminates respectively
can be compared by their {\em distinctive power} on similar graphs:
{\em $\bP(G; \bX)$ is at most as distinctive as $\bQ(G; \bY)$}, $\bP(G;\bX) \leq_{s.d.p} \bQ(G;\bY)$ 
if any two similar graphs $G,H$ which are
$\bQ$-equivalent are also $\bP$-equivalent.
{\em $\bP(G;\bX)$  and $\bQ(G; \bY)$ are s.d.p.-equivalent}, 
$\bP(G;\bX) \sim_{s.d.p} \bQ(G; \bY)$ if for any two similar graphs $G,H$
$\bP$-equivalence and
$\bQ$-equivalence coincide.
We can also compare graph polynomials on graphs without requiring similarity.
In this case we say that a graph polynomial $\bP$ is at most as distinctive as $\bQ$, $\bP \leq_{d.p.} \bQ$,
if for all graphs $G_1$ and $G_2$ we have that 
$$
\bQ(G_1; \bY) = \bQ(G_2; \bY) \mbox{   implies   }
\bP(G_1; \bX) = \bP(G_2; \bX)
$$
$\bP$ and $\bQ$ are d.p.-equivalent iff both
$\bP \leq_{d.p.} \bQ$ and
$\bQ \leq_{d.p.} \bP$.
D.p.-equivalence is stronger that s.d.p.-equivalence:

\begin{lemma}
\label{le:new}
For any two graph polynomials $\bP$ and $\bQ$ we have:
$\bP \leq_{d.p.} \bQ$ implies
$\bP \leq_{s.d.p.} \bQ$. 
\end{lemma}

A graph $G$ is $\bP$-unique if for all graphs $G'$ the polynomial identity
$\bP(G;\mathbf{X})=\bP(G';\mathbf{X})$ implies that
$G$ is isomorphic to $G'$.
As a graph invariant $\bP(G;\mathbf{X})$ can be used to check whether two graphs are not isomorphic.
For $P$-unique graphs $G$ and $G'$ the polynomial $\bP(G;\mathbf{X})$ can also be used to 
check whether they are isomorphic.

Our notion of similarity is extracted from the literature on graph polynomials:
It is implicitly used frequently both in claims that two polynomials are 
``really the same'', or ``the same up to a prefactor''.
From a logical point of view one would rather define a more general notion: Let $\Sigma$ be a finite
set of graph parameters. Two graphs $G,H$ are $\Sigma$-similar if they have the same values $s(G)=s(H)$ for
all $s \in \Sigma$.
It is easy, but currently of little use, to rewrite the definitions of various forms of equivalence
of graph polynomials using $\Sigma$-similarity
rather than similarity as we defined it in this paper.

\begin{theorem}
\label{th:dp}
\label{th:equiv}
\begin{enumerate}[(i)]
\item
$\bP$ is at most as distinctive as $\bQ$, $\bP \leq_{d.p} \bQ$, iff there is a function
$F: \Z[\bY] \rightarrow \Z[\bX]$ such that for every graph $G$ we have
$$
\bP(G;\bX) = F(\bQ(G;\bY))
$$
\item
$\bP$ is at most as distinctive as $\bQ$ on similar graphs, $\bP \leq_{s.d.p} \bQ$, iff there is a function
$F: \Z[\bY] \times \Z^3 \rightarrow \Z[\bX]$ such that for every graph $G$ we have
$$
\bP(G;\bX) = F(\bQ(G;\bY), n(G), m(G), k(G))
$$
\item
Furthermore, both for d.p. and s.d.p., if both $\bP$ and $\bQ$ are computable, then $F$ is computable, too.
\end{enumerate}
\end{theorem}

The equivalence in
(ii) in Theorem \ref{th:equiv} was  first proved in 
\cite{ar:MakowskyRavveBlanchard2014}. For the convenience of the reader we repeat it below.
Moreover, (iii) is new, and follows from our definition of computability of graph polynomials.
We note here that (ii) is useful for proving  d.p.-reducibility, whereas (i) is more
useful to prove its negation.
Theorem \ref{th:equiv} shows that our definition of d.p.-equivalence of graph polynomials
is mathematically equivalent to the definition proposed in
\cite{ar:MerinoNoble2009}.

\begin{proof}[Proof of Theorem \ref{th:equiv}(ii)-(iii)]
\ \\
(ii)
$\Rightarrow$:\\
Let $S$ be a set of finite graphs and
$s \in \Z[\bX]$. 
For a graph polynomial $\bP$  we define:
\begin{gather}
\bP[S] = \{ s \in \Z[\bX]: \bP(G)=s \mbox{ for some } G \in S\} \notag\\
\bP^{-1}(s)= \{G : \bP(G)=s\}. \notag
\end{gather}
Now assume $\bP(G;\bX) \preceq_{s.d.p.} \bQ(G;\bY)$. 
\\
If $\bQ^{-1}(s) \neq \emptyset$, then for every $G_1, G_2 \in \bQ^{-1}(s)$
we have 
$\bQ(G_1, \bY) = \bQ(G_2, \bY)$, and therefore
$\bP(G_1, \bX) = \bP(G_2, \bX)$. 
Hence $P[\bQ^{-1}(s)] = \{t_s\}$ for some $t_s \in \Z[\bX]$.
Now we define 
\begin{gather}
F_{\bP,\bQ}(s)=
\begin{cases}
t_s & \bQ^{-1}(s) \neq \emptyset\\
s & \mbox{else}
\end{cases}
\notag
\end{gather}

$\Leftarrow$: \\
Assume there is a function
$F: \Z[\bY] \rightarrow \Z[\bX]$ such that 
for all graphs $G$ we have 
$ F(\bQ(G)) = \bP(G)$.

Now let $G_1, G_2$ be similar graphs such that 
$\bQ(G_1) = \bQ(G_2)$.
Hence, $F(\bQ(G_1)) = F(\bQ(G_2))$.
Since for all $G$ we have $ F(\bQ(G)) = \bP(G)$,
we get $\bP(G_1)= \bP(G_2)$. 

Proof of (iii):
Now assume both 
$\bP(G;\bX)$ and
$\bQ(G;\bY)$ are computable.
To see that $F$ is computable we note that it suffices, as in the proof of (i),
to  find an element $s$ in the range of  $\bQ$  and a graph  $G_s$ such that $\bQ(G_s)=s$.
The latter can be done since the range of $\bQ$ is Turing decidable by Definition \ref{def:computable}(ii).
\end{proof}

\subsection{Examples of equivalent graph polynomials}
\label{se:equiv-ex}
\begin{example}
\label{ex:matching}
Let $m_k(G)$ denote the number of $k$-matchings ($k$ many independent edges) of $G$.
There are two versions of the univariate matching polynomial, \cite{bk:LovaszPlummer86}:
The {\em matching defect polynomial} (or {\em acyclic polynomial})
$$
\mu(G;X)= \sum_{k=0}^{\lfloor \frac{n}{2} \rfloor} (-1)^k m_k(G) X^{n-2k},
$$
and the
{\em matching generating polynomial}
$$
g(G;X)= \sum_{k=0}^n m_k(G) X^k.
$$
The relationship between the two is given by
\begin{gather}
\mu(G;X)=
\sum_{k=0}^{\lfloor \frac{n}{2} \rfloor}
(-1)^k m_k(G) X^{n-2k} =
X^{n}
\sum_{k=0}^{\lfloor \frac{n}{2} \rfloor}
(-1)^k m_k(G) X^{-2k} =
\notag
\\
=X^{n}
\sum_{k=0}^{\lfloor \frac{n}{2} \rfloor}
 m_k(G) ((-1)\cdot X^{-2})^{k}
=
X^{n}
\sum_{k=0}^{\lfloor \frac{n}{2} \rfloor}
m_k(G) (-X^{-2})^{k} =
X^ng(G;(-X^{-2}))
\notag
\end{gather}
It follows that $g$ and $\mu$ are equally distinctive, 
and can be computed  from each other by a simple substitution and multiplication
of a factor which depends only on the number of vertices, edges and connected components.
However, assume $G$ has no isolated vertices and $E_n$ is the graph without edges on $n$ vertices
with $E_0 = \emptyset$. Then $g(G;X) = g(G \sqcup E_n;X)$ for all $n \in \N$, i.e.,
$g(G;X)$ is invariant under addition or removal of isolated vertices.
However, this is not true for $\mu(G;X)$, which depends on  the number of isolated vertices.
\end{example}

\begin{example}
\label{ex:coefficients-1}
Let $\bP(G;X)$ be a univariate graph polynomial with integer coefficients and
\begin{gather}
\bP(G;X) = \sum_{i=0}^{d(G)} a_i(G) X^i = 
\notag \\
= \sum_{i=0}^{d(G)} b_i(G) X_{(i)} =
\notag \\
=\sum_{i=0}^{d(G)} c_i(G) {X \choose i} =
\notag \\
= \prod_i^{d(G)} (X - z_i)
\notag
\end{gather}
where $X_{(i)} = X(X-1)\cdot \ldots \cdot (X-i+1)$ is the falling factorial function.
We denote by  $a\bP(G)=(a_0(G), a_1(G), \ldots a_{d(G)})$,
$b\bP(G) =(b_0(G), b_1(G), \ldots b_{d(G)})$ 
and
$c\bP(G) =(c_0(G), c_1(G), \ldots c_{d(G)})$ 
the coefficients of these polynomial presentations
and by $z\bP(G) = (z_1, \ldots , z_{d(G)})$ the roots of these polynomials with their multiplicities.
We note that the four presentations of $\bP(G;X)$
are all d.p.-equivalent.
\end{example}

\begin{example}
\label{ex:spectral-1}
Let $G=(V(G), E(G))$ be a loopless graph without multiple edges.
Let $A_G$ be the adjacency matrix of $G$, $D_G$ the diagonal matrix with $(D_G)_{i,i} =d(i)$, the degree of the vertex $i$,
and $L_G = D_G -A_G$.
In spectral graph theory two graph polynomials are considered, the {\em characteristic polynomial of $G$},
here denoted by $P_A(G;X) = \det (X\cdot \mathbb{I}- A_G)$, and the {\em Laplacian polynomial}, 
here denoted by $P_L(G;X) = \det (X\cdot \mathbb{I}- L_G)$. Here $\mathbb{I}$ denotes the unit element in
the corresponding matrix ring.
Here we show that the polynomials $P_{A}(G;X)$ and $P_L(G;X)$ 
are d.p.-incomparable.
$G$ and $H$ in Figure \ref{fig1}  are similar.
We have 
$$P_{A}(G;X)=P_{A}(H;X)= (X-1)(X+1)^2(X^3 -X^2 -5X +1),$$ but $G$ has eight spanning trees, and $H$ has six.
Therefore, $P_L(G;X) \neq P_L(H;X)$, as one can compute the number of spanning trees from $P_L(G;X)$.
For more details, cf. \cite[Exercise 1.9]{bk:BrouwerHaemers2012}.

\begin{figure}[h]
\begin{center}
\includegraphics[width=.15\linewidth]{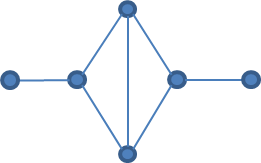} 
\hspace{2cm}
\includegraphics[width=.07\linewidth]{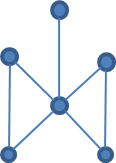} 
\\
$G$ \hspace{4cm} $H$
\end{center}
\caption{Similar graphs with different number of spanning trees}
\label{fig1}
\end{figure}

On the other hand, $G'$ and $H'$ in Figure \ref{fig2} are similar, but
$G'$ is not bipartite, whereas, $H'$ is.
\begin{figure}[h]
\begin{center}
\includegraphics[width=.15\linewidth]{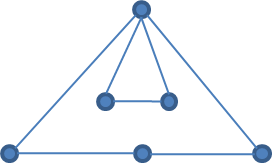} 
\hspace{2cm}
\includegraphics[width=.15\linewidth]{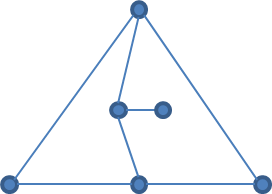} 
\\
$G'$ \hspace{4cm} $H'$
\caption{Similar graphs with different chromatic numbers}
\label{fig2}
\end{center}
\end{figure}
%
%
%
%
%
Hence $P_A(H';X) \neq P_A(G',X)$, but
$P_L(H';X)=P_L(G';X)$.
See, \cite[Lemma 14.4.3]{bk:BrouwerHaemers2012}.

{\bf Conclusion:} The characteristic polynomial and the Laplacian polynomial
are d.p.-incomparable.
However, if restricted to $r$-regular graphs, they are d.p.-eqivalent, \cite{bk:BrouwerHaemers2012}.
\end{example}

\subsection{Prefactor equivalence}
We recall that a graph parameter $f(G)$ with values in some function space $\bF$ over some ring $\cR$ 
is called a 
{\em similarity function} if for any two similar graphs $G,H$ we have that $f(G)=f(H)$.
If $\bF$ is  a subset of the set of analytic functions we speak of
{\em analytic similarity functions}.

If $\bF$ is the polynomial ring $\Z[\bX]$ with
set of indeterminates $\bX= (X_1, \ldots X_r)$,
we speak of {\em  similarity polynomials}.
It will be sometimes useful to allow classes of functions spaces 
which are closed under
{\em reciprocals} and {\em inverses} rather
than just similarity polynomials.  

\begin{example}
Typical examples of similarity functions
are 
\begin{enumerate}
\item
The nullity $\nu(G) = m(G)- n(G) +k(G)$ and the rank $\rho(G) =n(G)- k(G)$ of a graph $G$
are similarity polynomials with integer coefficients.
\item
Similarity polynomials can be formed inductively starting with
similarity functions $f(G)$
not involving indeterminates, and
monomials of the form $X^{g(G)}$, where $X$ is an indeterminate and 
$g(G)$ is a 
similarity function
not involving indeterminates.
One then closes under pointwise addition, subtraction, multiplication and substitution of
indeterminates $X$ by similarity polynomials.
\item
$f(G; X) = n(G)X^2$ is a similarity polynomial with integer coefficients.
Its inverse $f^{-1}(G; X) = n(G)^{-1} X^{\frac{1}{2}}$ is analytic at any point $a\in \R$ with $a \neq 0$.
Its reciprocal $\frac{1}{f(G;X)}$ is rational.
\end{enumerate}
\end{example}

In the literature one often wants to say that two graph polynomials are {\em almost the same}.
We propose a definition which makes this precise.

%
\begin{definition}
Let 
$\bP(G; Y_1, \ldots, Y_r)$  and
$\bQ(G; X_1, \ldots X_s)$ 
be two multivariate graph polynomials with coefficients in a ring $\cR$.
\begin{enumerate}
\item
We say that $\bP(G;\bY)$ is {\em prefactor reducible to $\bQ(G;\bX)$ 
over a set of similarity functions $\bF$},
and we write 
$$
\bP(G;\bY) \preceq_{prefactor}^{\bF} \bQ(G;\bX)
$$ 
if there are similarity functions $f(G;\bY)$ and $g_i(G; \bX), i \leq r$ in $\bF$
such that
$$
\bP(G;\bY) = f(G;\bY) \cdot \bQ(G; g_1(G; \bY), \ldots, g_r(G;\bY))
$$
\item
We say that $\bP(G;\bY)$ is {\em substitution reducible to} $\bQ(G; \bX)$ over $\bF$
and we write 
$$
\bP(G;\bY) \preceq_{subst} \bQ(G;\bX)
$$ 
if $f(G;\bX)=1$ is the constant function for all graphs $G$.
\item
We say that
$\bP(G;\bY)$  and
$\bQ(G;\bX)$  are {\em  prefactor equivalent},
and we write
$$
\bP(G;\bY) \sim_{prefactor} \bQ(G; \bX)
$$ 
if the relation holds in  both directions.
\item
Substitution equivalence
$ \bP(G; \bY) \sim_{subst} \bQ(G;\bX)$ 
is defined analogously.
\end{enumerate}
\end{definition}

The following properties follow from the definitions.
\begin{proposition}
Assume we have two graph polynomials
$\bP(G; \bY)$  and
$\bQ(G; \bX)$. 
For reducibilities we have:
\begin{enumerate}
\item
$ \bP(G; \bY) \preceq_{subst} \bQ(G; \bX) $ implies
$ \bP(G; \bY) \preceq_{prefactor} \bQ(G; \bX) $.
\item
$ \bP(G; \bY) \preceq_{prefactor} \bQ(G; \bX) $ implies
$ \bP(G; \bY) \preceq_{s.d.p.} \bQ(G; \bX) $.
\end{enumerate}
The corresponding implications for equivalence obviously also hold.
\end{proposition}

\subsection{The classical examples}
\begin{example}[The universal Tutte polynomial]
Let $T(G;X,Y)$ be the Tutte polynomial, \cite[Chapter 10]{bk:Bollobas98}.
The universal Tutte polynomial is defined by
$$
U(G;X,Y,U,V,W) = U^{k(G)} \cdot V^{\nu(G)} \cdot W^{\rho(G)} \cdot T\left(G; \frac{UX}{W}, \frac{Y}{U}\right).
$$
$U(G;X,Y,U,V,W)$ is the most general graph polynomial satisfying the recurrence relations
of the Tutte polynomial in the sense that every other graph polynomial satisfying these recurrence
relations is a substitution instance of $U(G;X,Y,U,V,W)$.

Here, $\nu(G) = m(G) -n(G) + k(G)$ is the nullity of $G$, and $\rho(G) = n(G) -k(G)$ is the rank of $G$.
Clearly, $U(G;X,Y,U,V,W)$ is 
prefactor equivalent to $T(G;X,Y)$ using rational similarity functions.
\end{example}

\begin{example}[The matching polynomials]
In Example \ref{ex:matching} we have already seen two
versions of the matching polynomials:
\begin{gather}
\mu(G;X) = \sum_{i=0}^{\lfloor \frac{n}{2} \rfloor} (-1)^i m_i(G) X^{n(G)- 2i}
\notag \\
g(G;Y) = \sum_{i=0}^{\lfloor \frac{n}{2} \rfloor} m_i(G) Y^i
\notag 
\end{gather}
The original definition in \cite{ar:HeilmannLieb72} is a bivariate version of the matching polynomial:
\begin{gather}
M(G; X,Y) = \sum_{i=0}^{\lfloor \frac{n}{2} \rfloor} m_i(G) X^i Y^{n(G)- 2i}
\notag
\end{gather}
We have
$ \mu(G;X) = X^{n(G)} \cdot g(G; -X^{-2}) $
and 
$M(G;X,Y) = Y^{n(G)} \cdot g(G; \frac{X}{Y^2})$.
Clearly, all three matching polynomials are mutually 
prefactor bi-reducible using analytic similarity functions.
\end{example}

\begin{example}
The following graph polynomials are d.p.-equivalent but incomparable by prefactor reducibility:
\begin{enumerate}
\item
$M(G;X)$ and $M(G;X)^2$;
\item
$\mu(G;X)$ and $\sum_i m_i(G){{X}\choose{i}}$.
\end{enumerate}
\end{example}
In the literature there are at least two theorems which state that
two graph polynomials have the same coefficients if restricted to some
graph class $\mathcal{K}$.

\begin{theorem}[C.D Godsil, I. Gutman, \cite{ar:GodsilGutman81}]
\label{th:GodsilGutman}
\label{th:chebyshev}
Let
$\mu(G;X)$  be
the defect matching polynomial  and
$P_{A}(G;X)$
the characteristic polynomial. 
Let $\mathcal{F}$ be the  
class of forests. 
Then for every graph in $\mathcal{F}$
we have that $\mu(G;X)$ is d.p.-equivalent to $P_{A}(G;X)$ and even stronger, that
$$
\mu(G;X) =
P_{A}(G;X).
$$
\end{theorem}
Now let
$$M(G; X,Y) = \sum_i m_i(G) X^i Y^{n(G)- 2i}$$
be the bivariate matching polynomial, and
$$M'(G; X) = \sum_i m_i(G) X^i X^{n(G)- 2i}= \sum_i m_i(G) X^{n(G)- i}$$
its substitution instance for $Y=X$.
Furthermore let 
$$\hat{\chi}(G;X) = \chi(\bar{G};X)$$ 
be the chromatic polynomial of the complement graph of $G$.
\begin{remark}
\begin{enumerate}
\item
$M(G; X,Y)$ is d.p.-equivalent to $M'(G; X)$ using a simple substitution.
\item
$\chi(G;X)$ and $\hat{\chi}(G;X)$ are d.p.-incomparable.
To see this, we note that for any number $m \in \N$ the polynomial $\chi(G;X)$ evaluated at $X=m$ 
does not distinguish cliques of size
bigger than $m$, whereas $\hat{\chi}(G;X)$ evaluated at $X=m$ does distinguish between them.
\end{enumerate}
\end{remark}

\begin{theorem}[E.J. Farrell, E.G.  Whitehead, \cite{ar:FarrellWhitehead1992}]
\label{th:FarrellWhitehead}
Let  $\Delta$ be the 
class of triangle-free graphs.
Then for each $G \in \Delta$ we have
that $M'(G;X)$ is d.p.-equivalent to $\hat{\chi}(G;X)$ and even stronger, that
$$
M'(G;X) = \hat{\chi}(G;X).
$$
\end{theorem}
In both Theorems \ref{th:GodsilGutman} and \ref{th:FarrellWhitehead}  the equality of the polynomials 
says something about the particular presentation of the graph polynomials but not about properties
of the graphs.

\section{How to represent graph polynomials?}
\label{apal-represent}

\subsection{Choosing a basis in the polynomial ring}

\begin{example}
\label{ex:coefficients}
Let $d(G)$ be a graph parameter, and
let $\bP(G;X)$ be a univariate graph polynomial with integer coefficients.
\begin{enumerate}[(i)]
\item
Assume
\begin{gather}
\bP(G;X) = \sum_{i=0}^{d(G)} a_i(G) X^i = \sum_{i=0}^{d(G)} b_i(G) X_{(i)} =
\notag \\
= \sum_{i=0}^{d(G)} c_i(G) X^{(i)} = \sum_{i=0}^{d(G)} e_i(G) {X \choose i} =
\prod_i^{d(G)} (X - z_i(G)),
\notag
\end{gather}
where $$X_{(i)} = X(X-1)\cdot \ldots \cdot (X-i+1)$$ is the falling factorial function, 
$$X^{(i)} = X(X+1)\cdot \ldots \cdot (X+i)$$ is the rising factorial function,
and $z_i$ are its roots.
Clearly, these are different presentations of the same polynomial, hence they are all d.p.-equivalent.
\item
Now look at the polynomials below, where the coefficients remain the same, but the
polynomial basis is changed:
\begin{gather}
\bP(G;X) = P_0(G;X)= \sum_{i=0}^{d(G)} a_i(G) X^i  = \prod_i^{d(G)} (X - z_i(G)) \\
P_1(G;X)
= \sum_{i=0}^{d(G)} a_i(G) X_{(i)}\\ 
P_2(G;X)
= \sum_{i=0}^{d(G)} a_i(G) X^{(i)} \\
P_3(G;X)
= \sum_{i=0}^{d(G)} a_i(G) {X \choose i}
\end{gather}
Obviously, $P_i(G;X)$ are different polynomials which have different roots, but
by Theorem \ref{th:dp} they are all d.p.-equivalent.
\end{enumerate}
\end{example}

Example \ref{ex:coefficients} shows that the location of the roots of a graph polynomial
is not invariant under d.p.-equvalence.

The notion of d.p.-equivalence (having the same distinguishing power) 
of graph polynomials evolved very slowly, mostly in implicit arguments.
Originally, a graph polynomial such as the chromatic or characteristic polynomial
had a {\em unique} definition which both determined its algebraic presentation and its semantic content.
The need to spell out semantic equivalence emerged when the various forms of the 
Tutte polynomial had to be compared.
As it was to be expected, some of the presentations of the Tutte polynomial 
had more convenient properties than others,
and some of the properties of one form got completely lost when passing to another semantically equivalent form.

Two 
d.p.-equivalent
polynomials carry the same combinatorial information about the underlying graph,
independently of their presentation as polynomials. This situation is analogous to the
situation in Linear Algebra: Similar matrices represent the same linear operator under two different bases.
The choice of a suitable basis, however, may be useful for numeric evaluations.
Here d.p.-equivalent graph polynomials represent the same combinatorial information under two different
polynomial representations. The choice of a particular polynomial representation $\bP(G;\bX)$ may carry
more numeric information about a particular graph parameter $p(G)$ determined by $\bP(G;\bX)$.

\subsection{Typical forms of graph polynomials}
\label{se:gp}
In this subsection we look at six types of graph polynomials: generalized chromatic polynomials
and polynomials defined as generating functions of induced or spanning subgraphs, determinant polynomials,
and 
graph polynomials arising from generating functions of relations.
In Section \ref{se:many} we show that they are truly of different form and use this in order to justify
our choice of Second Order Logic $\SOL$ as a suitable formalism for definability of graph polynomials.

More precisely,
let $\cC$ be a graph property. 
\begin{description}
\item[Generalized chromatic:]
Let $\chi_{\cC}(G;k)$ denote the number of vertex colorings of $G$
with at most $k$ colors such that each color class induces a graph in $\cC$.
If we count instead of vertex colorings edge colorings then the color class consists of sets of edges,
which induce a spanning subgraph in $\cC$.
It was shown in \cite{ar:KotekMakowskyZilber08,ar:KotekMakowskyZilber11} that 
$\chi_{\cC}(G;k)$ is a polynomial in $k$
for any graph property $\cC$ both for vertex and edge colorings.
Polynomials of this form were introduced first in \cite{pr:Harary85}, and will be referred to in this paper
as Harary polynomials.

A further generalization of chromatic polynomials was introduced 
in \cite{ar:MakowskyZilber2006,ar:KotekMakowskyZilber11}. 
The definition is model theoretic and too complicated to be given here.
A typical example would be counting the number of rainbow colorings with at most $k$ colors,
which  are edge colorings such that any two vertices are connected by at least one path with all its edges
receiving different colors. We shall see in Section \ref{se:zilber} that there are such colorings which are not
Harary colorings.
Generalized chromatic polynomials are further studied in \cite{ar:GHKMN-2016}.

\item[Generating functions:]
Let $A \subseteq V(G)$ and $B \subseteq E(G)$. We denote by $G[A]$ the induced subgraph of $G$
with vertices in $A$, and by $G<B>$ the spanning subgraph of $G$ with edges in $B$.
\begin{enumerate}[(i)]
\item
Let $\cC$ be a graph property. Define
$$
P_{\cC}^{ind}(G;X)= \sum_{A \subseteq V: G[A] \in \cC} X^{|A|}.
$$
\item
Let $\cD$ be a graph property which is closed under adding isolated vertices, i.e.,
if $G \in \cD$ then $G \sqcup K_1 \in \cD$. Define
$$
P_{\cD}^{span}(G;X)= \sum_{B \subseteq E: G<B> \in \cD} X^{|B|}.
$$
\end{enumerate}
\item[Generalized Generating functions:]
Let $X_i: i \leq r$ be indeterminates and $f_i: i \leq r$ be graph parameters.
We also consider graph polynomials of the form
$$
P_{\cC, f_1, \ldots, f_r}^{ind}(G;X)= \sum_{A \subseteq V: G[A] \in \cC} \prod_{i=1}^r X_i^{f_i(G[A])}
$$
and
$$
P_{\cC, f_1, \ldots, f_r}^{span}(G;X)= \sum_{B \subseteq E: G<B> \in \cD} \prod_{i=1}^r X_i^{f_i(G<B>)}.
$$
\item[Determinants:]
Let $M_G$ be a matrix associated with a graph $G$, such as the adjacency matrix, the Laplacian, etc.
Then we can form the polynomial $\det(\mathbf{1} \cdot X - M_G)$, where $\mathbf{1}$ is the unit matrix
of the same size as the order of $G$.
\end{description}
Special cases are 
the chromatic polynomial $\chi(G;X)$, 
the independence  polynomial $I(G;X)$, 
the Tutte
polynomial $T(G;X,Y)$ and
the characteristic polynomial of a graph $p_{char}(G;X)$.
Note that, in the sense of the following subsection, 
$\chi(G;X)$, $I(G;X)$ and $p_{char}(G;X)$ are mutually d.p.-incomparable, and $\chi(G;X)$
has strictly less distinctive power than $T(G;X,Y)$.

In Section \ref{se:relation} we shall see that there are graph polynomials defined in the
literature which seemingly do not fit the above frameworks. 
This is the case for the usual definition of
the {\em generating matching polynomial}: 
$$
\sum_{M \subseteq E(G): \mbox{match}(M)} X^{|M|},
$$
where $\mbox{match}(M)$ says that $(V(G),M)$ is a matching, i.e., $M$ is a set of isolated edges in $G$.
However, we shall see in Section \ref{se:relation} that there is another definition of the same polynomial
which is a generating function. In stark contrast to this, we shall prove there,
that the dominating polynomial 
$$
DOM(G;X)  = \sum_{A \subseteq V(G): \Phi_{dom}(A)} X^{|A|},
$$ 
where $\Phi_{dom}(A)$ says that $A$ is a dominating set of $G$,
cannot be written as a generating function, (Theorem \ref{th:dominating}).
This motivates the next definition, see also Section \ref{se:relation}.

\begin{description}
\item[Generating functions of a relation]
Let $\Phi$ be a property of pairs $(G,A)$ where $G$ is a graph and $A \subseteq V(G)^r$
is an $r$-ary relation on $G$.
Then the generating function of $\Phi$ is defined by
$$
\bP_{\Phi}(G;X) =
\sum_{A \subseteq V(G)^r: \Phi(G,A)} X^{|A|}.
$$
\item[The most general graph polynomials]
Further generalizations of chromatic polynomials were studied
in
\cite{ar:MakowskyZilber2006,ar:KotekMakowskyZilber11,phd:Kotek}
and in \cite{ar:GarijoGoodallNesetril2013,GNOdM16a}.
In \cite{ar:MakowskyZilber2006,ar:KotekMakowskyZilber11}
it was shown that the most general graph polynomials can be obtained
using model theory as developed in \cite{Zilber-uct,ar:CherlinHrushovski}.
A similar approach was used in
\cite{ar:GarijoGoodallNesetril2013,GNOdM16} based on ideas
from \cite{ar:delaHarpeJaeger1995}.
However, for our presentation here, we do not spell out the details
of this approach.
\end{description}

\subsection{Syntactic vs semantic properties of graph polynomials}
\label{MRB:semantic}
An {\em  $n$-ary property of graph polynomials $\Phi$}, aka a GP-{\em property}, is a subset of
the set of graph polynomials
$\mathfrak{GP}_{\mathcal{R},m}^n$ in $\mathcal{R}[\bX]$ in $m$ indeterminates.
$\Phi$ is a {\em semantic property} if it is closed under d.p.-equivalence.
Semantic properties are independent of the  particular presentation of its members.
Consequently, we call a property $\Phi$, which does depend on the presentation of
its members, a {\em syntactic (aka algebraic) property}.
Let us make this definition clearer via examples:

\begin{Examples}
\begin{enumerate}
\item 
The GP-property which says that for every graph $G$ the polynomial $\bP(G,\bX)$ is $P$-unique,
is a semantic property.
\item 
The unary GP-properties of univariate graph polynomials that for each graph $G$ the polynomials
$\bP(G;X)$ is 
{\em monic}\footnote{A univariate polynomial is monic if the leading coefficient equals $1$.},
or that its coefficients are unimodal\footnote{A sequence of numbers $a_i: i \leq m$ is unimodal   
if there is $k \leq m$ such that 
$a_i \leq a_j$ for $i <j < k$ and
$a_i \geq a_j$ for $ k \leq i <j \leq m$. 
}, 
is not a semantic GP-property, because, by applying Theorem \ref{th:equiv},
multiplying each coefficient by a fixed integer gives a d.p.-equivalent graph polynomial.
\item 
The GP-property that
the multiplicity of a certain value $a$ as a root of $\bP(G;X)$ coincides with the value of
a graph parameter $p(G)$ with values in $\N$,
is not a semantic property.
For example, the multiplicity of $0$ as a root of the Laplacian polynomial is the number
of connected components $k(G)$ of $G$, 
\cite[Chapter 1.3.7]{bk:BrouwerHaemers2012}.
However, stating that for two graphs $G_1, G_2$ with $\bP(G_1;X)=\bP(G_2;X)$ we also have
$p(G_1)=p(G_2)$, is a semantic property.
\item 
Similarly, proving that the leading coefficient of a univariate graph polynomial $\bP(G;X)$ 
equals the number of vertices of $G$ 
is not a semantic property, for the same reason.
However, proving that two graphs $G_1, G_2$ with 
$\bP(G_1;\mathbf{X}) = \bP(G_2;\mathbf{X})$ have the same number of vertices is semantically meaningful.
\item 
In similar vain, the classical result of \cite{ar:GodsilGutman81}, 
that the characteristic polynomial of a forest equals the 
(acyclic) matching polynomial of the same forest,
is a syntactic coincidence, or reflects a clever choice in the definition of the acyclic matching polynomial, 
but it is not a semantic GP-property.
The semantic GP-property of this result says that if we restrict our graphs to forests, 
then the characteristic and the
matching polynomials (in all its versions) have the same distinctive power on trees of the same size.
We discussed this and similar examples in 
\cite{ar:MakowskyRavveBlanchard2014} and paraphrase them again in Section \ref{se:zeros}. 
\end{enumerate}
\end{Examples}
To prove a semantic GP-property it is sometimes easier to prove
a stronger non-semantic version. From the above examples,
(iii), (iv) and (v)
are illustrative cases for this.

To motivate our definition of d.p.-equivalence we  first
give the examples taken from \cite{ar:MakowskyRavveBlanchard2014}.
For the multivariate case we use 
\cite{ar:AverbouchKotekMakowskyRavve2011} which shows that the universal EE-polynomial
$\xi(G,X,Y,Z)$ and the component polynomial $C(G;X,Y,Z)$ from \cite{phd:Trinks} are d.p.-equivalent and
are comparable and more expressive than the Tutte polynomial, the matching polynomials,
the independent set polynomial, and the chromatic polynomial.

\newif\ifrevised
\revisedtrue
\section{Distinctive power of various presentations of graph polynomials}
\label{se:linalg}
\label{se:dp}
We already know that s.d.p.-equivalent graph polynomials
can have very different forms, which do not reflect properties of the graphs.
In this section we do restrict the graph polynomials to be of a specific form
$\bP_{\cC}^{form}(G;\bX)$ where $\cC$ is a graph property,
and $\bP_{\cC}^{form}(G;\bX)$ is a presentation  of the graph polynomial
which is uniquely determined by $\cC$, as is the case for Harary polynomials, or generating
functions of induced or spanning subgraphs.
We then ask two questions: 
\begin{enumerate}[(i)]
\item
How does the choice of $\cC$ affect the distinctive power of $\bP_{\cC}^{form}(G;\bX)$,
and
\item
could different choices of $\cC$ yield graph polynomials of the same distinctive power?
\end{enumerate}

\subsection{s.d.p.-equivalence and d.p-equivalence of graph properties}
We recall that
a class of graphs $\cS$ which consists of all graphs having the same number of 
vertices, edges and connected components
is called a {\em similarity class}.

A graph polynomial $\cC$ with values in $\Z$ (without indeterminates) is a graph property 
if $\cC(G) =0$ or $\cC(G)=1$. In this we say $G \in \cC$ iff $\cC(G)=1$.
Let $\cC$ be a graph property, 
two graphs $G,H$ are $\cC$-equivalent if either both are in $\cC$
or both are not in $\cC$. 
We denote by $\bar{\cC}$ the graph property $\cG - \cC$.

Therefore we have:
\begin{proposition} 
\label{pr:dp-properties}
\label{th:char-0}
\begin{enumerate}[(i)]
\item
Two
graph properties
$\cC_1$ and $\cC_2$ are d.p.-equivalent iff 
either  $\cC_1 = \cC_2$ or $\cC_1  = \bar{\cC_2}$.
\item
Two
graph properties
$\cC_1$ and $\cC_2$ are s.d.p.-equivalent iff for every similarity class $\mathcal{S}$
either  $\cC_1 \cap \cS = \cC_2 \cap \cS$ or $\cC_1 \cap \cS  = \bar{\cC_2} \cap \cS$.
\end{enumerate}
\end{proposition} 
\begin{proof}
It is straightforward that
if $\cC_1$ and $\cC_2$ are d.p.-equivalent then $\cC_2 \cap \cS = \cC_1 \cap \cS$ or $\cC_2  \cap \cS= \bar{\cC_1} \cap \cS$.
\\
For the other direction, we prove first that 
$\cC_1 \cap \cS \subseteq \cC_2 \cap \cS$ or $\cC_1 \cap \cS \subseteq \bar{\cC_2} \cap \cS$.
\\
By a symmetrical argument, we then prove also
$\cC_2 \subseteq \cC_1$ or $\cC_2 \subseteq \bar{\cC_1}$,
$\bar{\cC_1} \subseteq \cC_2$ or $\bar{\cC_1} \subseteq \bar{\cC_2}$ and
$\bar{\cC_2} \subseteq \cC_1$ or $\bar{\cC_2} \subseteq \bar{\cC_1}$.
Now the result follows.
\end{proof}
\begin{remark}
If $\cC_1$ and $\cC_2$ are s.d.p.-equivalent it is possible that for a similarity class $\cS$ we have
$\cC_1 = \cC_2$ but for another similarity class $\cS'$ we have
$\cC_1 = \bar{\cC_2}$.
\end{remark}

\begin{proposition}
\label{prop:dp-incomparable}
\begin{enumerate}[(i)]
\item
Let $\cC_1$ and $\cC_2$ be two graph properties.
Assume  that
both $\cC_1$ and $\cC_1 $ are not empty and do not contain all finite graphs,
and that $\cC_1  \neq \cC_2 $ and $\cC_1  \neq \bar{\cC_2}$.
Then $\cC_1$ and $\cC_2$ are  s.d.p.-incomparable, i.e., $\cC_1 \not \leq_{d.p.} \cC_2$ and $\cC_2 \not \leq_{d.p.} \cC_1$.
\item
Let $\cC_1$ and $\cC_2$ be two graph properties.
Assume there is a similarity class $\cS$ such that
both $\cC_1 \cap \cS$ and $\cC_1 \cap \cS$ are not empty and do not contain all finite graphs in $\cS$,
 and that $\cC_1 \cap \cS \neq \cC_2 \cap \cS$ and $\cC_1  \cap \cS\neq \bar{\cC_2} \cap \cS$.
Then $\cC_1$ and $\cC_2$ are  s.d.p.-incomparable, i.e., $\cC_1 \not \leq_{s.d.p.} \cC_2$ and $\cC_2 \not \leq_{s.d.p.} \cC_1$.
\end{enumerate}
\end{proposition}
\begin{proof}
We prove only (i) and leave the proof of (ii) to the reader.
Assume 
$G_1 \in (\cC_1 - \cC_2) \cap \cS$, 
$G_2 \in (\cC_2 - \cC_1) \cap \cS$ and
$G_3 \in \cC_1 \cap \cC_2 \cap \cS$, the other cases being similar.
Then $G_2, G_3 \in \cC_2 \cap \cS$. 
If $\cC_1 \leq_{d.p.}$, we would have that 
both $G_2, G_3 \in \cC_1 \cap \cS$, or
both $G_2, G_3 \not \in \cC_1 \cap \cS$, a contradiction.
\end{proof}

In the next two subsections we look at graph polynomials, which are either generating functions, or count colorings
which, in both cases, solely depend on a graph property $\cC$.

\subsection{Graph polynomials as generating functions}

Let $\cC$ be a graph property, and $\cD$ be a graph property closed under adding and removing isolated vertices.
Recall from Section \ref{se:gp} the definitions
$$
\bP_{\cC}^{ind}(G;X)= \sum_{A \subseteq V: G[A] \in \cC} X^{|A|}
\mbox{\ \ \ \    and   \ \ \ \ }
\bP_{\cD}^{span}(G;X)= \sum_{B \subseteq E: G\langle B \rangle \in \cD} X^{|B|}.
$$
Let
$|V(G)|= n(G)$ and
$|E(G)|= m(G)$. 

\begin{proposition}
\begin{enumerate}[(i)]
\item
$\cC \leq_{d.p.} \bP_{\cC}^{ind}(G;X)$ and
\item
$\cD \leq_{d.p.} \bP_{\cD}^{span}(G;X)$. 
\end{enumerate}
\end{proposition}
\begin{proof}
(i) follows from the fact that
$G \in \cC$ iff the coefficient of $X^{n(G)}$ in $\bP_{\cC}^{ind}(G;X)$ does not vanish.
\\
Similarly, (ii) follows from the fact that
$G \in \cC$ iff the coefficient of $X^{m(G)}$ in $\bP_{\cC}^{span}(G;X)$ does not vanish.
\end{proof}

From Lemma \ref{le:new} we get immediately:
\begin{corollary}
\begin{enumerate}[(i)]
\item
$\cC \leq_{s.d.p.} \bP_{\cC}^{ind}(G;X)$ and
\item
$\cD \leq_{s.d.p.} \bP_{\cD}^{span}(G;X)$. 
\end{enumerate}
\end{corollary}

\begin{proposition}
\label{prop:complement}
With
$|V(G)|= n(G)$ and
$|E(G)|= m(G)$ we have:
\begin{enumerate}[(i)]
\item
$\bP_{\cC}^{ind}(G;X) + \bP_{\bar{\cC}}^{ind}(G;X) = 
(1+X)^{n(G)}$
\item
$\bP_{\cD}^{span}(G;X) + \bP_{\bar{\cD}}^{span}(G;X) = 
(1+X)^{m(G)}$
\end{enumerate}
\end{proposition}
\begin{proof}
(i):
Put
$$
c_i(G) = | \{ A \subseteq V(G): |A|=i, G[A] \in \mathcal{C} \}|
$$
and
$$
\bar{c}_i(G) = | \{ A \subseteq V(G): |A|=i, G[A] \not\in \mathcal{C} \}|.
$$
Clearly, 
$$
c_i(G) + \bar{c}_i(G) = { n(G) \choose i},
$$
hence
$$
\sum_{i=0}^{n(G)} \left(c_i(G) + \bar{c}_i(G)\right) X^i = (1+X)^{n(G)}.
$$

(ii) is similar, but we need that for a set of edges $A \subseteq E(G)$ the spanning subgraph
$G\langle A\rangle =(V(G),A) \in \cD$ iff  $V(A),A) \in \cD$,
where 
$$
V(A) = \{ v \in V(G) : \mbox{ there is   } u \in V(G) \mbox{ with  } (u,v) \in A\}.
$$
\end{proof}

\ifrevised
\begin{proposition}
\label{th:char-1}
Let $\cC_1$ and $\cC_2$, $\cD_1$ and $\cD_2$ be graph properties 
such that
$\cC_1$ and $\cC_2$  and $\cD_1$ and $\cD_2$ are pairwise d.p.-equivalent,
\begin{enumerate}[(i)]
\item
\label{th:char-1.1}
$\bP_{\cC_1}^{ind}(G;X)$ and $\bP_{\cC_2}^{ind}(G;X)$
are s.d.p.-equivalent;
\item
\label{th:char-1.2}
If, additionally, $\cD_1$ and $\cD_2$ are closed under the addition and removal of isolated vertices, then
$\bP_{\cD_1}^{span}(G;X)$ and $\bP_{\cD_2}^{span}(G;X)$
are s.d.p.-equivalent;
\end{enumerate}
\end{proposition}
\begin{proof}
We prove  only (i), (ii) is proved analogously.
\\
(i): We use Proposition 
\ref{pr:dp-properties}.
If $\cC_1 =  \cC_2$, clearly, 
$\bP_{\cC_1}^{ind}(G;X) = \bP_{\cC_2}^{ind}(G;X)$, 
hence they are d.p.-equivalent.
If $\cC_1 =  \bar{\cC_2}$, we use Proposition \ref{prop:complement} together with 
Proposition \ref{th:dp}. But Proposition \ref{prop:complement} depends on the $n(G)$, 
hence we get only that 
$\bP_{\cC_1}^{ind}(G;X) = \bP_{\cC_2}^{ind}(G;X)$, 
are d.p.-equivalent.
\end{proof}

A sequence of polynomials $f(\bX)_i$ is C-finite if it satisfies a linear recurrence relation
of depth $p$
with constant coefficients $a_j(\bX) \in \Z[\bX]$:
\begin{gather}
f(\bX)_{i+p} = \sum_{j=0}^{p-1} a_j(\bX) f(\bX)_{i+j}.
\end{gather}
Let 
$G_n$ be an indexed sequence of graphs  such that
the sequence of polynomials $X^{|V(G_n)|}$ is C-finite.
This assumption is true for 
the sequences $C_n$, $P_n$, $K_n$, of cycles, paths and cliques, and all sequences $G_n$ of graphs
provided the function $|V(G_n)|$ is linear in $n$.
In particular, it applies to Theorem \ref{th:chebyshev}.
We shall now show that C-finiteness of the sequences of polynomials $\bP_{\cC}^{ind}(G_n;X)$
of Theorem \ref{th:chebyshev}
is a semantic property graph polynomials as generating functions.
However, the particular form of the recurrence relation is not.

\begin{theorem}
\label{th:char-C-finite}
Let $\cC_1$ and $\cC_2$, $\cD_1$ and $\cD_2$ be graph properties 
such that
$\cC_1$ and $\cC_2$  and $\cD_1$ and $\cD_2$ are pairwise d.p.-equivalent,
and let $G_n$ be an indexed sequence of graphs. 
Furthermore, assume that the sequence of polynomials $X^{|V(G_n)|}$ is C-finite.
Then
\begin{enumerate}[(i)]
\item
$\bP_{\cC_1}^{ind}(G_n;\bX)$ is C-finite iff $\bP_{\cC_2}^{ind}(G_n;\bX)$ is C-finite. 
\item
$\bP_{\cD_1}^{span}(G_n;\bX)$ is C-finite iff $\bP_{\cD_2}^{span}(G_n;\bX)$ is C-finite; 
\end{enumerate}
\end{theorem}
\begin{proof}
This follows in both cases from the fact that the sum and difference of two C-finite sequences
is again C-finite together with Proposition \ref{prop:complement}.
\end{proof}

\else
\begin{theorem}
\label{th:char-1}
Let $\cC_1$ and $\cC_2$, $\cD_1$ and $\cD_2$ be graph properties 
such that
$\cC_1$ and $\cC_2$  and $\cD_1$ and $\cD_2$ are pairwise d.p.-equivalent,
 and let $G_n$ be an indexed sequence of graphs. Then
\begin{enumerate}[(i)]
\item
\label{th:char-1.1}
$\bP_{\cC_1}^{ind}(G;X)$ and $\bP_{\cC_2}^{ind}(G;X)$
are d.p.-equivalent;
\item
$\bP_{\cC_1}^{ind}(G_n;\bX)$ is C-finite iff $\bP_{\cC_2}^{ind}(G_n;\bX)$ is C-finite. 
\item
\label{th:char-1.2}
If, additionally, $\cD_1$ and $\cD_2$ are closed under the addition and removal of isolated vertices, then
$\bP_{\cD_1}^{span}(G;X)$ and $\bP_{\cD_2}^{span}(G;X)$
are d.p.-equivalent;
\item
$\bP_{\cD_1}^{span}(G_n;\bX)$ is C-finite iff $\bP_{\cD_2}^{span}(G_n;\bX)$ is C-finite; 
\end{enumerate}
\end{theorem}
\begin{proof}
We prove (i) and (ii), (iii) and (iv) are proved analogously.
\\
(i): We use Proposition 
\ref{pr:dp-properties}.
If $\cC_1 =  \cC_2$, clearly, 
$\bP_{\cC_1}^{ind}(G;X) = \bP_{\cC_2}^{ind}(G;X)$, hence they are d.p.-equivalent.
If $\cC_1 =  \bar{\cC_2}$, we use Proposition \ref{prop:complement} together with 
Proposition \ref{th:dp}. 
\\
(ii) follows from Proposition \ref{prop:complement}.
\end{proof}
\fi 

\subsection{Harary polynomials}
Recall from the introduction the definition of
$\chi_{\cC}(G;k)$ as the number of colorings of $G$
with at most $k$ colors such that each color class induces a graph in $\cC$.
\begin{theorem}[J. Makowsky and B. Zilber, cf. \cite{ar:KotekMakowskyZilber11}]
$\chi_{\cC}(G;k)$ is a polynomial in $k$
for any graph property $\cC$.
\end{theorem}

\ifrevised
\else
\begin{theorem}
\label{th:char-2}
Let $\cC_1$ and $\cC_2$ two d.p.-equivalent graph properties.
Then
$\chi_{\cC_1}(G;k)$ and $\chi_{\cC_2}(G;k)$
are d.p.-equivalent. 
\end{theorem}
\begin{proof}
This follows from the fact that
$G \in \cC$ iff $\chi_{\cC}(G;1) =1$.
\end{proof}
\fi 

In contrast to Proposition \ref{prop:complement}
the relationship between $\chi_{\cC}(G;k)$ and $\chi_{\bar{\cC}}(G;k)$
is not  at all obvious.
What can we say about $\chi_{\bar{\cC}}(G;k)$ in terms of $\chi_{\cC}(G;k)$?

\begin{proposition}
\label{prop:char-2}
There are two classes $\cC_1$ and $\cC_2$ which are d.p.-equivalent but such that
$\chi_{\cC_1}$ and $\chi_{\cC_2}$ are not d.p.-equivalent.
\end{proposition}
\begin{proof}
Let $\cC_1$ be all the disconnected graphs and
let $\cC_2$ be all the connected graphs.
As they are complements of each other, they are d.p.-equivalent.
\\
We compute for $K_i$:
$$
\chi_{\cC_1}(K_i; j) =0, j \in \N^+
$$
because there is no way to partition $K_i$ into any number of disconnected  parts.
Hence $\chi_{\cC_1}(K_i; X) =0$.
$$
\chi_{\cC_2}(K_i; 2) = 2^i -2
$$
because every partition of $K_i$ into two nonempty parts gives two connected graphs.
Therefore $\chi_{\cC_2}$ distinguishes between cliques of different size,
whereas $\chi_{\cC_1}$ does not.
\end{proof}

We note, however, that the analogue of Proposition \ref{th:char-1} for Harary polynomials
remains open.

\subsection{d.p.-equivalence of graph polynomials}
The converse of Theorem
\ref{th:char-1}(\ref{th:char-1.1}) and (\ref{th:char-1.2}) is not true:

\begin{proposition}
\label{prop:counterexample}
There are graph properties $\cC_1$ and $\cC_2$ which are not d.p.-equivalent,
but such that
\begin{enumerate}[(i)]
\item
$\bP_{\cC_1}^{ind}(G;X)$ and $\bP_{\cC_2}^{ind}(G;X)$
are d.p.-equivalent.
\item
$\chi_{\cC_1}(G;X)$ and $\chi_{\cC_2}(G;X)$
are d.p.-equivalent.
\end{enumerate}
\end{proposition}
\begin{proof}
For (i)
Let $\cC_1 = \{K_1\}$ and $\cC_2 =\{K_2, E_2\}$
where $E_n$ is the graph on $n$ vertices and no edges.
We compute: 
\begin{gather}
\bP_{\cC_1}^{ind}(G;X) = n(G)\cdot X \notag \\
\bP_{\cC_2}^{ind}(G;X) =  {n(G) \choose 2} \cdot X^2. \notag
\end{gather}
\ifrevised
For (ii) we
choose $\cC_1=\{K_1\}$ as before, but $\cC_2=\{K_1,K_2,E_2\}$. 
\\
{\bf Claim 1}: \\
$\chi_{\cC_2}(G,X)\leq_{d.p.} n(G)$
Proof: Let $G_1$ and $G_2$ be two graphs with the same number of vertices. 
W.l.o.g. assume they have the same vertex set $V(G_1)=V(G_2)=V$.
Now notice for every $f:V \to [k]$, $f$ is a $\cC_2$-coloring of $G_1$ iff it is a 
$f$ is a $\cC_2$-coloring of $G_2$.
Hence $\chi_{\cC_2}(G_1,X)=\chi_{\cC_2}(G_2,X)$ whenever $G_1$ and $G_2$ have the same number of vertices. 
\\
{\bf Claim 2}: \\ $n(G) \leq_{d.p.} \chi_{\cC_2}(G,X)$
Proof: First denote for every $m$,
$n_{even}(m) = \prod_{i=0}^{m-1} \binom{2(m-i)}{2}$
and 
$n_{odd}(m) = \prod_{i=0}^{m-1} \binom{2(m-i)+1}{2}$.
For every graph $G$, there is a natural number $m(G)$ such that $n(G)=2m(G)$ or $n(G)=2m(G)+1$.
If $n(G)=2m(G)$, $\chi_{\cC_2}(G,m(G))=n_{even}(m(G))$.
If $n(G)=2m(G)+1$, $\chi_{\cC_2}(G,m(G))=n_{odd}(m(G))$.  
Note $n_{odd}(r) > n_{even}(r)$ for every natural number $r$.
The minimal natural number $r$ such that $\chi(G,r)>0$ is equal to $m(G)$. 
We get that the minimal $r$ such that $\chi_{\cC_2}(G,r)>0$ determines $n(G)$. 
Hence $\chi_{\cC_1} =_{d.p.} \chi_{\cC_2}$. 
\else
For (ii) we compute: 
\begin{gather}
\chi_{\cC_1}(G;X) =  {X \choose n}\cdot n! \notag \\
\chi_{\cC_2}(G;X) = 
\begin{cases}
\prod_{i=0}^m {2(m -i) \choose 2} \cdot {X \choose m} n!, & n(G)=2m  \\
0, & n(G) = 2m+1
\end{cases}
\notag 
\end{gather}
$\bP_{\cC_1}^{ind}(G;X)$ and $\bP_{\cC_2}^{ind}(G;X)$,
and $\chi_{\cC_1}(G;X)$ and $\chi_{\cC_2}(G;X)$,
are d.p.-equivalent, 
because  for both
two graphs $G_1, G_2$ get pairwise the same value of the polynomials iff $n(G_1)=n(G_2)$.
\fi 
\end{proof}
We leave it to the reader to construct the corresponding counterexample for
$\bP_{\cD}^{span}(G;X)$.

We cannot use Proposition \ref{prop:dp-incomparable} to show that there
infinitely many d.p.-incomparable graph polynomials of the form
$\bP_{\cC}^{ind}(G;X)$. However, we can construct explicitly 
infinitely many d.p.-incomparable graph polynomials of this form.

\section{Choosing the appropriate formalism for graph polynomials}
\label{se:many}
\subsection{Motivation}

In this section we show 
that, up to s.d.p.-equivalence, there are uncountably mutually incomparable
graph polynomials which are
generating functions of
counting induced or spanning subgraphs or Harary colorings of graph properties.
This suggests that the graph properties defining these polynomials have to be restricted.

However, although
many classical graph polynomials from the literature are either 
generating functions of
counting induced or spanning subgraphs or Harary colorings,
we show
that the characteristic and Laplacian polynomials are not of this form. 
This also holds for
other naturally defined
graph polynomials.
This suggest that our framework has to be extended.

In Section \ref{se:sol} we finally introduce the graph polynomials definable in Second Order Logic $\SOL$
as the suitable formalism.

\subsection{Many d.p.-inequivalent graph polynomials}

For the rest of this section,
let 
$C_i$ be the undirected cycle on $i$ vertices, and $C_i^*$ the graph which consists of a copy of $C_{i-1}$
together with a new vertex $v$ which is connected to exactly one of the vertices of $C_{i-1}$.
Clearly, $C_i$ and $C_i^*$ are similar.
Furthermore, let
$\cC_i = \{C_i\}$, and let $G_i^k$ consist of the disjoint union of $k$-many copies of $C_i$,
and let $\hat{G}_i^k$ consist of the disjoint union of $k-1$ copies of $C_i^*$ together
with one copy of $C_i$.
Again, $\hat{G}_i^k$ and $G_i^k$ are similar.

We compute:
\begin{lemma}
\label{incomp}
\begin{gather} 
\bP_{\cC_{j}}^{ind}(G_{i}^{k};X) = \bP_{\cC_{j}}^{ind}(\hat{G}_{i}^{k};X) = 0 \mbox{ for } i \neq j, i \neq j+1, \tag{i} \\ 
\bP_{\cC_{i}}^{ind}(G_{i}^{k};X) = k \cdot  X^{i} \tag{ii}  \\
\bP_{\cC_{i}}^{ind}(\hat{G}_{i}^{k};X) =  X^{i} \tag{iii}  
\end{gather}
\end{lemma}

\begin{theorem}
\label{th:dp-incomparable}
For all $i,j$ with $i \neq j$ and $i \neq j+1$ the polynomials 
$\bP_{\cC_{i}}^{ind}$ and
$\bP_{\cC_{j}}^{ind}$ are d.p.-incomparable, hence
there are infinitely many d.p.-inequivalent graph polynomials of the form $\bP_{\cC}^{ind}(G;X)$.
\end{theorem}

\begin{proof}
Assume $i,j \geq 3$ with $i \neq j$ and $i \neq j+1$.
We first prove 
$\bP_{\cC_{i}}^{ind} \not <_{d.p.} \bP_{\cC_{j}}^{ind}$ for $i \neq j$ and $i \neq j+1$.
\\
We look at the graphs $G_j^2$ and $\hat{G}_j^2$.
$\bP_{\cC_{j}}^{ind}(G_{j}^{2};X) = 2 \cdot X^i$ by Lemma \ref{incomp}(ii).
$\bP_{\cC_{j}}^{ind}(\hat{G}_{j}^{2};X) = X^{i}$ by Lemma \ref{incomp}(iii).
Hence,
$\bP_{\cC_{j}}^{ind}$ distinguishes between the two graphs $G_j^2$ and $\hat{G}_j^2$.
However,
$\bP_{\cC_{i}}^{ind}(G_{j}^{2};X) = 
\bP_{\cC_{i}}^{ind}(\hat{G}_{j}^{2};X) = 0$, by Lemma \ref{incomp}(i).
Hence,
$\bP_{\cC_{i}}^{ind}$ does not distinguish between the two graphs.

To prove
$\bP_{\cC_{j}}^{ind} \not <_{d.p.} \bP_{\cC_{i}}^{ind}$ for $j \neq i$ and $j \neq i+1$,
we look at the graphs $G_i^2$ and $\hat{G}_i^2$.
In this case 
$\bP_{\cC_{j}}^{ind}$ does not distinguish between the two graphs $G_i^2$ and $\hat{G}_i^2$,
but $\bP_{\cC_{i}}^{ind}$ does.
\end{proof}

\begin{theorem}
\label{th:dp-incomparable-span}
There are infinitely many d.p.-inequivalent graph polynomials of the form $\bP_{\cC}^{span}(G;X)$.
\end{theorem}
\begin{proof}
The proof mimics the proof of Theorem \ref{th:dp-incomparable} with following changes:
Instead of $\cC_i$ we use $\cD_i = \{ C_i \sqcup E_j: j \in \N \}$ 
and
\begin{gather}
\bP_{\cD_{j}}^{span}(G_{i}^{k};X) = 0 \mbox{ for } i \neq j, i \neq j+1, \notag \\
\bP_{\cD_{i}}^{span}(G_{i}^{k};X) = k \cdot X^{i}. \notag \\
\bP_{\cD_{j}}^{span}(\hat{G}_{i}^{k};X) = 
\begin{cases}
0 &  i \neq j, i \neq j+1 \\
(k-1) \cdot X^j & i=j+1
\end{cases}
\notag \\
\bP_{\cD_{i}}^{span}(\hat{G}_{i}^{k};X) =  X^{i}. \notag 
\end{gather}
\end{proof}

Next we look at Harary polynomials $\chi_i(G;X) = \chi_{\cC_i}(G;X)$.
We use the following obvious lemma:
\begin{lemma}
\label{le:chrom}
\begin{enumerate}[(i)]
\item
For $X = \lambda \in \N$:
$$
\chi_i(G_i^k;\lambda) =
\begin{cases}
\lambda_{(k)} &  \lambda \geq k \\
0 & \mbox{ else }
\end{cases}
$$
\item
$$
\chi_j(G_i^k, \lambda) = 0 
$$
provided that $i \neq j$.
\item
$$
\chi_j(\hat{G}_i^k, \lambda) = 0 
$$
provided that $k \geq 2$ or  $k=1, i \neq j$.
\end{enumerate}
\end{lemma}

\begin{theorem}
\label{th:dp-incomparable-chrom}
For all $i \neq j$ the polynomials 
$\chi_i$ and $\chi_j$
are d.p.-incomparable, hence
there are infinitely many d.p.-incomparable graph polynomials of the form $\chi_{\cC}$.
\end{theorem}

\begin{proof}
$\chi_i \not \leq_{d.p.} \chi_j$:
\\
We look at the graphs $G_i^2$ and $\hat{G}_i^2$.
By Lemma \ref{le:chrom} $\chi_j$ does not distinguish between $G_i^2$ and $\hat{G}_i^2$.
However, $\chi_i$ distinguishes between them.

To show that
$\chi_j \not \leq_{d.p.} \chi_i$,
we look at the graphs $G_j^2$ and $\hat{G}_j^2$.
By Lemma \ref{le:chrom} $\chi_i$ does not distinguish between $G_j^1$ and $G_j^2$.
However, $\chi_j$ does distinguish between them.
\end{proof}

\subsection{Generating functions of a relation}
\label{se:relation}
If, instead of counting induced (spanning) subgraphs with a certain graph property $\cC$ ($\cD$),
we count $r$-ary relations with a property $\Phi(A)$, we get a generalization of both
the generating functions of induced (spanning) subgraphs.
Here the summation is
defined by
$$
\bP_{\Phi}(G;X) = \sum_{A \subseteq E(G): \Phi(A)} X^{|A|}.
$$

For example,
the generating matching polynomial, defined as
$$
m(G;X) = \sum_{A \subseteq E(G): \Phi_{match}(A)} X^{|A|}
$$
can be written as
$$
m(G;X) = \sum_{A \subseteq E(G): G\langle A \rangle \in \cD_{match}} X^{|A|}
$$
with  $\cD_{match}$ being the disjoint union of isolated vertices and isolated edges.

However, not every graph polynomial
$\bP_{\Phi}(G;X)$
can be written as a generating function of induced (spanning) subgraphs.

Consider the graph polynomial
$$
DOM(G;X)  = \sum_{A \subseteq V(G): \Phi_{dom}(A)} X^{|A|},
$$ 
where $\Phi_{dom}(A)$ says that $A$ is a dominating set of $G$.

We compute:
\begin{gather}
DOM(K_2,;X) = 2X +X^2  \label{dom1},\\
DOM(E_2,;X) = X^2 \label{dom2}.
\end{gather}

\begin{theorem}
\label{th:dominating}
\label{th:dom-not-gen}
\begin{enumerate}[(i)]
\item
There is no graph property $\cC$ such that
$$
DOM(G;X) = \bP_{\cC}^{ind}(G;X).
$$
\item
There is no graph property $\cD$ such that
$$
DOM(G;X) = \bP_{\cD}^{span}(G;X).
$$
\end{enumerate}
\end{theorem}
\begin{proof}
(i):
Assume, for contradiction, there is such a $\cC$, and that
$K_1 \in \cC$.
The coefficient of $X$ in 
$\bP_{\cC}^{ind}(E_2;X)$ is $2$ because $K_1 \in \cC$. 
However, the coefficient of $X$ in $DOM(E_2;X)$ is $0$, by equation
(\ref{dom2}), a contradiction.

Now, assume $K_1 \not \in \cC$.
The coefficient of $X$ in 
$\bP_{\cC}^{ind}(K_2;X)$ is $0$, because $K_1 \not \in \cC$.
However, the coefficient of $X$ in $DOM(K_2;X)$ is  $2$, by equation
(\ref{dom1}), another contradiction.

(ii):
Assume, for contradiction, there is such a $\cD$.
The coefficient of $X$ in 
$\bP_{\cD}^{span}(K_2;X)$ is  $\leq 1$, because $K_2$ has only one edge. 
However, the coefficient of $X$ in $DOM(K_2;X)$ is  $2$, by equation
(\ref{dom1}), a contradiction.
\end{proof}

We can use Equation (\ref{dom1}) also to show the following:

\begin{theorem}
\label{th:notchrom}
\label{th:dom-not-chrom}
There is no graph property $\cC$ such that
$$
DOM(G;X) = \chi_{\cC}(G;X).
$$
\end{theorem}

\begin{proof}
First we note that $\chi_{\cC}(G;1)=1$ iff $\chi_{\cC}(G;1) \neq 0$ iff $G \in \cC$.
\\
Assume that $K_2 \in \cC$.  Then we have, using Equation (\ref{dom1}),
$$
\chi_{\cC}(K_2;1) = 1 = DOM(K_2,1) =3,
$$
a contradiction.
\\
Assume that $K_2  \not \in \cC$. Then we have, using Equation (\ref{dom1}),
$$
\chi_{\cC}(K_2;1) = 0 = DOM(K_2,1) =3,
$$
another contradiction.
\end{proof}

\ifskip\else
\subsection{Determinant polynomials}
There are only two matrices associated with graphs which have been used to define
graph polynomials: the adjacency matrix and the Laplacian.
The two resulting determinant polynomials are d.p.-incomparable.
It is conceivable to to define other matrix presentations of graphs,
and ask when they give rise to d.p.-equivalent determinant polynomials. 
The characterization and recognition problem in this case amounts to the question
when the characteristic polynomial of a matrix is the
the characteristic polynomial arising from a graph.
However, in this paper we do not pursue this further.

\subsection{Characterizing d.p.-equivalence for special classes of graph polynomials}
Theorems \ref{th:char-1} and Proposition \ref{prop:char-2} and Proposition \ref{prop:counterexample}
show that d.p.-equivalence of $\cC$ and $\cC_1$, respectively $\cD$ and $\cD_1$, is not enough
to characterize d.p.-equivalence of generating functions or Harary polynomials
defined by $\cC$ and $\cD$.
Sometimes d.p.-equivalence of graph properties only implies s.d.p.-equivalence 
of the corresponding graph polynomials.

\begin{problem}
Characterize d.p.-equivalence of graph polynomials arising from
$\cC$ and $\cD$ as
\begin{enumerate}
\item
Harary polynomials;
\item
generating functions of induced or spanning subgraphs;
\item
generating functions of relations.
\end{enumerate}
\end{problem}

\marginpar{Should we add something?}
\begin{framed}

Should we add something like this:

\begin{theorem}[Kind of....]
Given two graph polynomials 
$\bP_1(G;X)$ and 
$\bP_2(G;X)$, and a family of graphs
$G_n$ such that on $(G_n)_{n \in \N}$ 
$\bP_1(G;X)$ and 
$\bP_2(G;X)$ are d.p.-equivalent, \
and the sequence $\bP_1(G_n;X)$
is C-finite.
then there is a graph polynomial $\bP_3$ d.p.-equivalent to  $\bP_2(G;X)$
and the sequence $\bP_3(G_n;X)$
is C-finite.
\end{theorem}

There is a bit of freedom what we have to require on the special form of
the three graph polynomials $\bP_i(G;X)$ for $i=1,2,3$.
\end{framed}
\fi 

\subsection{Determinant polynomials}

For convenience of the reader we repeat the definition from Example \ref{ex:spectral-1}.
We assume that $V(G) = \{ 1, \ldots, n(G)\}$.

Let 
$$
A(G) = (a_{i,j}(G))= \begin{cases} 1 &  (i,j) \in E(G)\\
0 & \mbox{   else   }
\end{cases}
$$

$$
D(G) = (d_{i,j}(G))= \begin{cases} d_i &  (i,i) \in E(G)\\
0 & \mbox{   else   }
\end{cases}
$$
where $d_i(G)$ is the degree of the vertex $i$.

Finally $L(G) = D(G) - A(G)$.

The characteristic polynomial 
$P_A(G;X)$ is given by
$$
P_A(G;X) = \det(\mathbf{1} \cdot X - A(G))
$$
and the Laplacian polynomial
$P_L(G;X)$ is given by
$$
P_L(G;X) = \det(X \cdot D(G) - A(G)).
$$

The two resulting determinant polynomials, $P_A(G;X)$ and $P_L(G;X)$, are s.d.p.-incomparable, 
as shown by Figures \ref{fig1} and \ref{fig2} from Example \ref{ex:spectral-1}.

\begin{theorem}
\label{th:char-not-gen}
\begin{enumerate}[(i)]
There is no graph property $\cC$, and no graph property $\cD$ closed under isolated vertices, such that
\item
$P_A(G;X) = P_{\cC}^{ind}(G;X)$,
\item
$P_L(G;X) = P_{\cC}^{ind}(G;X)$,
\item
$P_A(G;X) = P_{\cD}^{span}(G;X)$,
\item
$P_L(G;X) = P_{\cD}^{span}(G;X)$,
\item
$P_A(G;X) = \chi_{\cC}(G;X)$, or
\item
$P_L(G;X) = \chi_{\cC}(G;X)$.
\end{enumerate}
\end{theorem}
\begin{proof}
We first compute:
\begin{gather}
P_A(E_1;X) = X, P_A(E_2;X) = X^2  \label{eq:char}\\
P_L(E_2;X) = 0, P_L(K_2,X) = X^2-1 \label{eq:lap}
\end{gather}
(i) and (v):
If $E_1 \not \in \cC$,
then $P_{\cC}(E_1;X) = 0$ and $\chi_{\cC}(E_1;1)=0$, 
which contradicts Equation \ref{eq:char}.
Otherwise, $E_1 \in \cC$, then 
$P_{\cC}(E_2;X) = X +Q(X)$ and $\chi_{\cC}(E_2;1)=0$, 
which again contradicts Equation \ref{eq:char}.
\\
(ii):
If $K_2 \not \in \cC$, then  
$P_{\cC}(K_2;X)=0$ 
which contradicts Equation \ref{eq:lap}.
Otherwise, $K_2 \in \cC$, then 
$P_{\cC}(K_2;X)=X^2$, 
which again contradicts Equation \ref{eq:lap}.
\\
(vi):
If $K_2 \in \cC$, then 
$\chi_{\cC}(K_2;1)=1$,
which contradicts Equation \ref{eq:lap}.
If $K_2 \not \in \cC$, then 
we distinguish two subcases: $K_1 \in \cC$, then 
$\chi_{\cC}(K_2;2)=2$,
and if
$K_1 \not\in \cC$, then 
$\chi_{\cC}(K_2;2)=0$.
However, 
$P_L(K_2,2) =3$ by  Equation \ref{eq:lap},
which again gives a contradiction.
\\
(iii) and (iv):
Assume $K_2 \not\in \cD$, then $P_{\cD}(K_2;X) =0$ 
which contradicts Equations \ref{eq:char} and \ref{eq:lap}.
Otherwise, assume
$K_2 \in \cD$, then $P_{\cD}(K_2;X) =X$ 
which again contradicts Equations \ref{eq:char} and \ref{eq:lap}.
\end{proof}

\ifskip
\else
Next we show that we can write both $P_A$ and $P_L$ as a linear combination
of graph polynomials which almost look like generating functions.

We can write $P_A(G;X)$ and $P_L(G;X)$ as follows.
Let $A'(G) =(a_{(i,j)}') = X \mathbb{I} -A(G)$ 
and $L'(G) = (l_{(i,j)}' = X \cdot D(G)  -A(G)$.

\begin{gather}
P_A(G;X)
= 
\sum_{\pi \in S_{n(G)}} (-1)^{sign(\pi)} \prod a_{i, \pi(i)}' \notag 
\\
= 
\sum_{\pi \in S_{n(G)}, \mbox{even}} 
\prod a_{i, \pi(i)}'  -
\sum_{\pi \in S_{n(G)}, \mbox{odd}} \prod a_{i, \pi(i)}'  \notag 
\end{gather}
$P_L(G;X)$ is obtained by replacing $A'$ with $L'$.

Next we observe that
\begin{gather}
\prod a_{i, \pi(i)}' = 
\begin{cases} 
0 & \exists i (i \neq \pi(i) \mbox{  and  } (i, \pi(i) \not \in E(G)\\
1 & \forall i (i \neq \pi(i) \mbox{  and  } (i, \pi(i) \in E(G)\\
X^{fp(\pi) & \pi \mbox{  has } fp(\pi) \mbox{  fixpoints, and  }
\forall i (i \neq \pi(i) \mbox{  and  } (i, \pi(i) \in E(G)
\end{cases} \\
\prod l_{i, \pi(i)}' = 
\begin{cases} 
0 & \exists i (i \neq \pi(i) \mbox{  and  } (i, \pi(i) \not \in E(G)\\
1 & \forall i (i \neq \pi(i) \mbox{  and  } (i, \pi(i) \in E(G)\\
\prod_i (d_i \cdot X)
& i = \pi(i) \mbox{  and  }
\forall i (i \neq \pi(i) \mbox{  and  } (i, \pi(i) \in E(G)
\end{cases} 
\end{gather}
Furthermore $d_v = \sum_{u: (u,v) \in E(G)} 1$.


This now looks like the difference of two generating functions of a relation.

It is conceivable to define other matrix presentations of graphs,
and ask when they give rise to d.p.-equivalent determinant polynomials. 
The characterization and recognition problem in this case amounts to the question
when the characteristic polynomial of a matrix is the
the characteristic polynomial arising from a graph.
However, in this paper we do not pursue this further.

\fi 

\subsection{Generalized chromatic polynomials}
\label{se:zilber}

Here we show that not all generalized chromatic polynomials are Harary polynomials.

An $mcp$-coloring with at most $k$-colors is an edge coloring such that between any two vertices
there is at least one path where all the edges have the same color.
Let $\chi_{mcp}(G;k)$ be the number of $mcp$-colorings of $G$ with at most $k$ colors.
It follows again from \cite{ar:KotekMakowskyZilber08,ar:KotekMakowskyZilber11}
that $\chi_{mcp}(G;k)$ is a polynomial in $k$.
In fact, it can be written as
$$\chi_{mcp}(G;k) = \sum_{\ell = 0}^{m(G)} c_{mcp}(\ell, G) {X \choose \ell},$$
where $c_{mcp}(\ell, G)$ is the number of $mcp$-colorings with exactly $\ell$ colors.

\begin{theorem}
\label{th:mcp-not-harary}
There is no graph property $\cC$ such that $\chi_{mcp}(G;k) = \chi_{\cC}(G;k)$ is a Harary polynomial.
\end{theorem}

\begin{proof}
We first observe that 
$$
\chi_{mcp}(G;k) = 
\begin{cases} k & G \mbox{   connected   }\\
0 & \mbox{  otherwise  }
\end{cases}
$$
because  if a graph is connected all the edges have to be colored by the same color.

Now assume for contradiction that $\chi_{mcp}(G;k) = \chi_{\cC}(G;k)$. 
Then $\cC =\cC_0$ is the class of  graphs $G= H_G \sqcup E_{i(G)}$ which a disjoint union of a connected graph $H$
with a set of isolated vertices.
However 
$$
\chi_{\cC_0}(G \sqcup G;2) =
\begin{cases} 4 & H_G \mbox{   connected   }\\
0 & \mbox{  otherwise  }
\end{cases}
$$
but for $H_G =G$ connected we have  $\chi_{mcp}(G;2) =2$.
\end{proof}

\begin{theorem}
\label{th:mcp-not-gen}
There is no graph property $\cC$ such that 
$\chi_{mcp}(G;k) 
$
is a generating function of induced  or spanning subgraphs.
\end{theorem}
\begin{proof}
Let 
$$
tv(connected)(G) = 
\begin{cases} 1 & G \mbox{   connected   }\\
0 & \mbox{  otherwise.  }
\end{cases}
$$
Then  $\chi_{mcp}(G;X) = tv(connected)(G) \cdot X$.
Here $tv$ stands for {\em truth value}.
We know that
$$\chi_{mcp}(G;k) = \sum_{\ell = 0}^{m(G)} c_{mcp}(\ell, G) {X \choose \ell},$$
where $c_{mcp}(\ell, G)$ is the number of $mcp$-colorings with exactly $\ell$ colors.
Hence,  
$$
c_{mcp}(\ell,G) =
\begin{cases} 1 & G \mbox{   connected   }\\
0 & \mbox{  otherwise  }
\end{cases}.
$$
On the other hand, for $P_{\cC}^{ind}(G;k) = \sum_{\ell} c_{\cC}(\ell, G) X^{\ell}$
if $G \in \cC$ has more than one vertex, the coefficient 
$C_{\cC}(n(G),G)=1$, but
$C_{mcp}(n(G),G)=0$.
Similarly, 
if 
for $P_{\cC}^{ind}(G;k) = \sum_{\ell} d_{\cC}(\ell, G) X^{\ell}$,
$G \in \cC$ has more than one edge, the coefficient 
$C_{\cC}(m(G),G)=1$, but
$C_{mcp}(m(G),G)=0$.
\end{proof}

\section{$\SOL$-definability of graph polynomials}
\label{se:sol}

In this section we present the formalism of $\SOL$-definable graph polynomials.
The idea originated in 
\cite{ar:CourcelleMakowskyRoticsDAM,ar:Makowsky01} and was further developed in \cite{phd:Kotek}.
It is the logical framework which includes all the examples of graph polynomials
so far discussed in this paper.

$\SOL$-definable graph polynomials are given using a finite set of $\SOL$-formulas
$\phi_1, \ldots, \phi_s$ such that
replacing all the formulas $\phi_i: i \leq s$ by  logically equivalent formulas $\psi_i$
the resulting polynomial remains the same.

As a starting point,
$\SOL$-definable graph polynomials include
the generating functions of $\SOL$-definable
graph properties. 
As we have seen, the dominating polynomial, the characteristic and Laplacian polynomials, 
and the generalized chromatic polynomials are  not of this form (Theorems
\ref{th:dom-not-gen},
\ref{th:char-not-gen},
\ref{th:mcp-not-gen}
).
Even if we include the Harary polynomials of $\SOL$-definable graph properties,
we still note that the dominating polynomial is not of this form (Theorem \ref{th:dom-not-chrom}).
To accommodate all these examples, we will give an inductive definition of $\SOL$-definable
graph polynomials by imposing the following closure properties: 
\begin{enumerate}[(i)]
\item sums and products of graph polynomials,
\item summation over relations on graphs,
\item products over tuples of vertices or edges of graphs,
\item substitution of indeterminates by  algebraic terms involving indeterminates,
\item substitution of indeterminates by  graph polynomials.
\end{enumerate}

Given a graph polynomial $\bP(G;X)$ there are uncountably many d.p.-equivalent graph polynomials.
However, there will be only countably many $\SOL$-definable graph polynomials.
To show that a certain property of a graph polynomial $\mathcal{X}$ is not a semantic property, 
it suffices to find
a s.d.p. (d.p.)-equivalent graph polynomial which does not have property $\mathcal{X}$.
Allowing all s.d.p.-equivalent graph polynomials really misses the point.

We are really interested in semantic properties of graph polynomials in a specific prescribed form.
Let $\mathcal{F}$ be a family of graph polynomials, such as generalized chromatic polynomials,
generating functions of induced or spanning subgraphs, or $\SOL$-definable polynomials.

By restricting the graph polynomials under consideration to $\mathcal{F}$
we say a property of a polynomial is semantically meaningful on $\mathcal{F}$
if all graph polynomials $\bP \in \mathcal{F}$ which are  d.p.-equivalent (s.d.p.-equivalent)
share this property.

\begin{example}
Let $\mathcal{F}$  be the class of 
graph polynomials given by
generating functions of induced subgraphs of property $\cC$.
Let $\bP_{\cC}^{ind}$ and
$\bP_{\cD}^{ind}$ 
from Section \ref{se:dp}.
Let $\cD = \bar{\cC}$ be the complement of $\cC$.
By Proposition \ref{th:char-1} we have that 
$\bP_{\cC}^{ind}$ and
$\bP_{\cD}^{ind}$ are s.d.p. equivalent.
For $G \in \cC$ we have that
$\bP_{\cC}^{ind}(G;X)$  is monic, but by Proposition \ref{prop:complement},
$\bP_{\cD}^{ind}(G;X)$ is not necessarily monic.

Hence, monic is not a semantic property even for 
graph polynomials restricted to
generating functions of induced subgraphs.
\end{example}

The framework of $\SOL$-definable graph polynomials allows us to analyze the graph theoretic (=semantic)
content of graph polynomials. To show that a property $\mathcal{X}$ of graph polynomials
is not a graph theoretic property it suffices to show:
\begin{quote}
For every $\SOL$-definable graph polynomial $\bP \in \mathcal{X}$ there is
a d.p.- or s.d.p.-equivalent $\SOL$-definable graph polynomial $\bQ  \not \in \mathcal{X}$ 
which can be constructed from the definition of $\bP$.
\end{quote}
Usually, $\bQ$ is explicitly given from the formulas defining $\bP$.
In some cases 
$\bQ$ is obtained from $\bP$ using substitutions and, possibly, by adding a prefactor.


\subsection{Second Order Logic}
We assume the reader is familiar with Second Order Logic.
Let $\tau$ be a finite set of relation symbols, i.e., 
a purely relational vocabulary.
We have individual variables $v_i$ and relation variables $U_{\rho(i), i}$ of arity
$\rho(i)$. The set of $\SOL(\tau)$-formulas is defined inductively.
We define atomic formulas over $\tau$ and equality using the relation symbols from $\tau$
and the relation variables. We close under Boolean connectives, and existential and universal
quantification over individual variables and relation variables.
We denote $\SOL(\tau)$-formulas with Greek letters, $\phi,\psi,\theta$, possibly with indices.
\\
{\bf Caveat:}
We use  here $m,n,k,$ here as summation indices and not as graph parameters.
\\
Let 
$\bU = (U_{\rho(1),1}, \ldots, U_{\rho(n),n})$ ,
$\bv = (v_{1}, \ldots, v_{m})$. 
We write
$\phi(\bU, \bv)$ for the formula with the indicated free variables.
Given a $\tau$-structure $\fA$ with universe $A$, 
relations $B_{\rho(i),i} \subseteq A^{\rho(i)}$ for $i\in [n]$
and $b_i \in A$ for $i \in [m]$, we put
$\bB = (B_{\rho(1),1}, \ldots, B_{\rho(n),n})$ ,
$\bb = (b_{1}, \ldots, b_{m})$  and we write
$\phi(\bB, \bb)$ for the formula evaluated over $\fA$. 

\subsection{Definable graph polynomials}
The notion of {\em definability of graph parameters and graph polynomials in $\SOL$} 
was first introduced\footnote{
In 
\cite{ar:CourcelleMakowskyRoticsDAM} we also deal with definability in Monadic Second Order Logic ($\MSOL$),
but in this paper this distinction is of no use.
Note however, that the results also hold for $\MSOL$.
} in
\cite{ar:CourcelleMakowskyRoticsDAM} and extensively studied in 
\cite{ar:GodlinKotekMakowsky08,ar:FischerKotekMakowsky11,ar:KotekMakowskyZilber11,phd:Kotek,ar:GodlinKatzMakowsky12,pr:KotekMakowskyRavveSYNASC,pr:MakowskyKotekRavve2013,ar:KotekMakowsky-LMCS2014}.

Let $\cR \in \{\Z, \Q, \R, \C\}$ a ring or field.
Given a graph $G =(V,E)$ we define the set of interpreted terms $\SOLEVAL(G)$  in $\cR[\bX]$ inductively.
\begin{enumerate}
\item
Elements of $\cR[\bX]$ are in $\SOLEVAL(G)$.
\item
$\SOLEVAL(G)$ is closed under addition, subtraction and multiplication in $\cR[\bX]$.
\item
$\SOLEVAL(G)$ is closed under substitution of indeterminates by elements of $\cR[\bX]$.
\item
(Small sums and products)
If $t \in \SOLEVAL(G)$, and $\phi(\bv)$ is a formula of $\SOL(\tau)$ with 
individual variables $v_1, \ldots, v_{\rho}$
and non-displayed interpreted individual and relation parameters,
then
$$
\sum_{\bb \in V^{\rho}:\phi(\bb)} t 
$$
and
$$
\prod_{\bb \in V^{\rho}:\phi(\bb)} t 
$$
are interpreted terms in
$\SOLEVAL(G)$.
\item
(Large sums)
If $t \in \SOLEVAL(G)$, and $\phi(U)$ is a formula of $\SOL(\tau)$ with 
relation variable $U$ of arity $\rho$
and non-displayed interpreted individual and relation parameters,
then 
$$
\sum_{B \subseteq V^{\rho}:\phi(B)} t 
$$
is a term in $\SOLEVAL(G)$.
\item
An expression $t \in \SOLEVAL(G)$ defines for each graph uniformly a polynomial $t(G) \in \cR[\bX]$.
\item
A graph polynomial $P(G,\bX)$ is $\SOL$-definable if there is an expression $t \in\SOLEVAL(G)$
such that for each graph $G$ we have $t(G) = \bP(G;\bX)$.
\end{enumerate}
We first give examples where we use {\em small}, i.e., polynomial sized sums and products:
\begin{Examples}
\begin{enumerate}
\item
The cardinality of $V$ is $\MSOL$-definable by $$\sum_{v\in V} 1$$
\item
The number of connected components of a graph $G$, $k(G)$ is
$\MSOL$-definable by 
$$\sum_{C \subseteq V: \mathrm{component}(C)} 1,$$
where $\mathrm{component}(C)$ says that $C$ is a connected component.
Although them sum ranges over subsets of $V$, it is small, because there are at most $|V|$-many
connected components.
\item
The graph polynomial $X^{k(G)}$ is
$\MSOL$-definable 
by 
$$\prod_{c \in V: \mathrm{first-in-comp}(c)} X$$
if we have a linear order in the vertices  and
$\mathrm{first-in-comp}(c)$ says that $c$ is a first element in
a connected component.
\end{enumerate}
\end{Examples}

Now we give examples with possibly {\em large}, i.e., exponential sized sums:
\begin{Examples}
\begin{enumerate}
\item[(iv)]
The number of cliques in a graph is 
$\MSOL$-definable by 
$$\sum_{C \subseteq V: \mathrm{clique}(C)} 1,$$
where $\mathrm{clique}(C)$ says that $C$ induces a complete graph.
\item[(v)]
Similarly ``the number of maximal cliques''  
is $\MSOL$-definable by 
$$\sum_{C \subseteq V: \mathrm{maxclique}(C)} 1,$$
where $\mathrm{maxclique}(C)$ says that $C$ induces a maximal complete graph.
\item[(vi)]
The clique number of $G$, $\omega(G)$ 
is $\SOL$-definable by 
$$\sum_{C \subseteq V: \mathrm{largest-clique}(C)} 1,$$
where $\mathrm{largest-clique}(C)$ says that $C$ induces a maximal complete graph
of largest size.
\item[(vii)]
The clique polynomial of $G$ is $\SOL$-definable by
$$
\sum_{C \subseteq V: \mathrm{clique}(C)}  \prod_{v \in C} X.
$$
\end{enumerate}
\end{Examples}
Now here are some prominent graph polynomials which are easily seen to be $\SOL$-definable.
\begin{Examples}
\label{ex:prominent}
\begin{enumerate}
\item
\label{ex:spectral}
Let $G=(V(G), E(G))$ be a loopless graph without multiple edges.
Here we consider again 
the {\em characteristic polynomial of $G$}, $P_{A}(G;X)$,
and the {\em Laplacian polynomial}, $P_L(G;X)$ from Example \ref{ex:spectral-1}. 
To see that both $P_{A}(G;X)$ and $P_L(G;X)$ are $\SOL$-definable, we write them as 
a sum of two $\SOL$-definable polynomials in distinct indeterminates $X_1$ and $X_2$,
and then put $X_1=X$ and $X_2= (-1)\cdot X$.
Here we use that $\SOL$-definable polynomials are closed under substitution by elements of the
polynomial ring.

In other words, to express the determinant
$$
P_B(G;X) = \det (X\cdot \mathbf{1}- B)
$$
of a matrix $B$ dependent on $G$,
we write
$$
P_B(G;X_1, X_2) = P^{even}_B(X_1) + P^{odd}_B(X_2)
$$
where $P^{even}_B(X_1)$  sums over all even permutations and 
where $P^{odd}_B(X_2)$  sums over all odd permutations and then
put
$$
P_B(G;X) = P^{even}_B(X) + P^{odd}_B((-1)\cdot X).
$$
Now we can use this to show that 
$P_{A}(G;X) = \det (X\cdot \mathbf{1}- A_G)$ and
$P_L(G;X) = \det (X\cdot \mathbf{1}- L_G)$ are  substitution instances of bivariate 
$\SOL$-definable graph polynomials.
\item
Let 
$$a_i(G) = \mid \{ U \subseteq V : (G, U) \models \phi(U)  \mbox{ and } |U|=i \} \mid$$
be {\em uniformly defined} numeric graph parameters.
Then
$$
\sum_i a_i(G) X^i  = 
\sum_{U:\phi(u)} X^{|U|}
$$
is a the {\em generic} form of an {\em $\SOL$-definable graph polynomial}.
\begin{enumerate}
\item
If $\phi(U)$ says that $U$ is a set of edges which form a matching, we get the
{\em matching generating polynomial} $g(G;X)$.
\item
If $\phi(U)$ says that $U$ is a set of vertices which form an independent set, we get the
{\em independence polynomial} $I(G;X)$.
\end{enumerate}
\item
The {\em Potts model} is the partition function
$$
Z(G;X,Y) = \sum_{B \subset E(G)} X^{k[B]} Y^{|B|}
$$
with $k[B]$ is the number of connected components of the spanning subgraph generated by $B$.
$Z(G;X,Y)$ is $\SOL$-definable if an order on the vertices is present, using the closure properties and
the previous examples.
\item
The chromatic polynomial $\chi(G;X)$ is $\SOL$-definable, using closure under substitution and  the fact that
$$
\chi(G;X) =  Z(G;X,-1).
$$
\end{enumerate}
\end{Examples}
\begin{remark}
\label{rem:negative}
Negative coefficients may occur in $\SOL$-definable polynomials, however they do occur only
as a result of substitution of negative numbers for indeterminates.
In the above examples $Z(G;X,Y)$ has no negative coefficients, but $\chi(G;X) =  Z(G;X,-1)$ does.
\end{remark}
In general, to show that a graph polynomial is definable in $\SOL$ may be difficult.
For instance, counting the number of planar induced subgraphs uses Kuratowski's or Wagner's characterization
of planarity. We do not know a general method to show that a graph polynomial is not $\SOL$-definable.
To show that it is not $\MSOL$-definable one can use the method of connection matrices, \cite{ar:KotekMakowsky-LMCS2014}.

\subsection{Normal form of $\SOL$-definable graph polynomials}

In \cite{ar:KotekMakowskyZilber11,phd:Kotek} the following normal form theorem was proved:

\begin{theorem}[Normal Form Theorem]
\label{th:normal-form}
Every $\SOL$-definable multivariate graph polynomial $\bP(G;\bX)$
can be written as
\begin{gather}
\bP(G;\bX)=
\sum_{A \subseteq V^{r_1}:\phi_1(A)} 
\ldots
\sum_{A \subseteq V^{r_s}:\phi_s(A)} 
\prod_{\bX_1 \in A:\psi_1(A, \bX_1)} X_1 
\cdot \ldots \cdot 
\prod_{\bX_t \in A:\psi_t(A, \bX_t)} X_t 
\notag
\end{gather}
with $\phi_i, i \leq s, \psi_j: j \leq t$ $\SOL$-formulas.
\end{theorem}

This shows that every $\SOL$-definable graph polynomial is a multiple generating function
of several $\SOL$-definable relations.

\subsection{Semantically equivalent presentations of graph polynomials}

Theorems \ref{th:GodsilGutman} and \ref{th:FarrellWhitehead}
show that, restricted to certain $\SOL$-definable graph classes
two different graph polynomials have identical polynomials as their values.
In fact, if we assume $\SOL$-definability, we can always achieve equality of the coefficients.

\begin{theorem}
\label{th:identical-coeff}
Assume $\mathcal{K}$ is a $\SOL$-definable graph class, and 
$\bP(G;\bX)$ and
$\bQ(G;\bX)$ are $\SOL$-definable
and d.p.-equivalent on $\mathcal{K}$.
\begin{enumerate}
\item
There is a $\SOL$-definable graph polynomial $P'(G;\bX)$ which is d.p-equivalent to $\bP(G;\bX)$ and
such that
for all $G \in \mathcal{K}$ we have that
$$ \bP'(G;\bX) =\bQ(G;\bX).$$
\item
If $\mathcal{K}$, $\bP(G;\bX)$ and $\bQ(G;\bX)$ are all computable 
(computable in exponential time),
so is $P'(G;\bX)$.
\end{enumerate}
\end{theorem}
\begin{proof}
(i):
We define
$$
\bP'(G;\bX) =
\begin{cases}
\bQ(G;\bX) \mbox{  if  } G \in \mathcal{K} \\
\bP(G;\bX) \mbox{  else  }. 
\end{cases}
$$
It is straightforward that 
$\bP(G;\bX)$ and
$P'(G;\bX)$ are d.p.-equivalent on all graphs, and satisfies the equality of the coefficients.
To see that 
$P'(G;\bX)$ is $\SOL$-definable we note that a case distinction given by a $\SOL$-definable class $\mathcal{K}$
is also $\SOL$-definable.

(ii):
The computability  of $P'(G;\bX)$ and its complexity statement follow  immediately 
from the computability  and complexity assumptions.
\end{proof}

\subsection{Consistency and the recognition problem}

Given a closed formula  $\phi$ in a logical system
consistency of $\phi$ asks whether there exists a structure $\fA$ such that
$\fA \models \phi$. If we consider $\hat{\phi}$ as a Boolean graph parameter, this can be expressed
as asking whether there is a graph $G$ such that $\hat{\phi}(G) =1$.

The generalization of consistency for $\cR$-valued graph parameters $\bP$ and a value $p \in \cR$ 
asks whether there is a graph $G$ such that $\bP(G) =p$.
In the case of the chromatic polynomial, H. Wilf in \cite{ar:Wilf1973} 
calls this the {\em recognition problem}.
We assume that H. Wilf had a constructive answer in mind which was not only algorithmic but 
algebraic and qualitative. If the parameter is $\bP(G;X) = X^{n(G)}$ 
the expected answer says: Given $p \in \Z[X]$ there is a graph $G$ such that 
$\bP(G;X)=p$ iff $p$ is monic and consists of
exactly one monomial $X^n$ with exponent $n \geq 1$.

In finite model theory consistency is computationally (recursively) enumerable,
and computable, provided there is a bound $b_{\phi} \in \N$ on the size
of the smallest model of $\phi$.

For a $\SOL$-definable graph polynomial  $\bP(G;\bX)$ such a bound always exists, hence the recognition problem
is decidable. To give a syntactic description of the form of $p = \bP(G;X)$, provided $G$ exists,
is a difficult problem, and wide open even for the case of the characteristic or the chromatic polynomial.
For a discussion  of the recognition problem, cf. \cite{ar:KotekMakowskyRavve2017arxiv}.

\subsection{Closure properties}

Here we look at closure properties of $\SOLEVAL$ under
reducibilities via distinctive power. 
Clearly, $\SOLEVAL$ is not closed under the relation $\leq_{d.p.}$ and $\leq{s.d.p.}$.
If $\bP(G;\bX)$ and $\bQ(G;\bY)$ are two graph polynomials with
$\bP(G;\bX) \leq_{d.p.} \bQ(G;\bY)$, and one of them is in $\SOLEVAL$, the other still may not be in $\SOLEVAL$.
However,
we have defined $\SOLEVAL$ to be closed under substitutions of indeterminates
by elements of the underlying polynomial ring. 
Hence we get:
\begin{proposition}
\label{pr:closure}
\begin{enumerate}[(i)]
\item
If $\bP(G;\bX) \preceq_{subst} \bQ(G;\bY)$ and
$\bQ(G;\bY) \in \SOLEVAL$, then
$\bP(G;\bY) \in \SOLEVAL$. 
\item
Let $\bP(G; X, \bY), \mathbf{R}(G;\bY) \in \SOLEVAL$ with indeterminates $X$ and $\bY$.
Then the result of substituting $\mathbf{R}(G;\bY)$ for $X$ in $\bP(G; X, \bY)$, $\bP(G; \mathbf{R}(G;\bY), \bY)$ 
is also in $\SOLEVAL$.
\item
If $\bP(G;\bX) \preceq_{prefactor} \bQ(G;\bY)$ using similarity functions in $\SOLEVAL$,
i.e.,
there are similarity functions $f(G; \bX), g_1(G; \bX), \ldots, g_{m_2}(G; \bX) \in \SOLEVAL$
such that
$
\bP(G; \bX) = f(G; \bX) \cdot \bQ(G; g_1(G; \bX), \ldots , g_{m_2}(G; \bX))
$
and $\bQ(G;\bY) \in \SOLEVAL$, then $\bP(G;\bY) \in \SOLEVAL$. 
\end{enumerate}
\end{proposition}
\begin{proof}
(i) follows from the definition of $\SOLEVAL$.
\\
(ii) is shown by induction on the definition of $\mathbf{R}(G;\bY)$.
\\
(iii) follows from (ii).
\end{proof}

\section{On the location of zeros of graph polynomials}
\label{se:zeros}
\subsection{Roots of univariate graph polynomials}
\label{csl-roots}
The literature on graph polynomials mostly got its inspiration from the successes in studying the  chromatic polynomial
and its many generalizations and the
characteristic polynomial of graphs. 
In both cases the roots of graph polynomials are given much attention and are meaningful when these polynomials model physical reality.

A complex number $z \in \C$ is a root  of a univariate graph polynomial $P(G;X)$
if there is a graph $G$ such that $P(G;z)=0$.
It is customary to study the
location of the roots of univariate  graph polynomials. 
Prominent examples, besides
the chromatic polynomial, the matching polynomial and the
characteristic polynomial and its Laplacian version,
are the independence polynomial, the domination polynomial and the vertex cover polynomial.

For a fixed univariate graph polynomial $P(G;X)$ typical statements about roots are:
\begin{enumerate}
\item
For every $G$ the roots of $P(G;X)$ are real.
This is the univariate version of stability or Hurwitz stability for real polynomials. 
It is true for the characteristic and the matching polynomial, 
\cite{bk:CvetkovicDoobSachs1995,bk:LovaszPlummer86}.
Similarly, for every claw-free graph $G$ the roots of the independence polynomial are real,
\cite{ar:ChudnovskySeymour2007roots}.
Incidentally, by a classical theorem of I. Newton, if all the roots of a polynomial 
with positive coefficients are real, then its coefficients are unimodal.
\item
Assuming that all roots of $P(G;X)$ are real,
the (second) largest root has an  interesting combinatorial interpretation.
This is true for the characteristic polynomial where the second largest eigenvalue is related
to the
Cheeger constant, 
\cite[Chapter 4]{ar:AlonMilman1985,bk:BrouwerHaemers2012}.
\item
The multiplicity of a certain value $a$ as a root of $P(G;X)$ has an interesting interpretation.
For example, the multiplicity of $0$ as a root of the Laplacian polynomial is the number
of connected components of $G$, 
\cite[Chapter 1.3.7]{bk:BrouwerHaemers2012}.
\item
For every $G$ all real roots of $P(G;X)$ are positive (negative)
or the only real root is $0$.
The real roots are positive in the case of the chromatic polynomial and the clique polynomial,
and negative for the independence polynomial, 
\cite{bk:DongKohTeo2005,ar:HoedeLi94,ar:BrownHickmanNowakowski2004,ar:GoldwurmSantini2000,phd:Hoshino}.
\item
For every $G$ the roots of $P(G;X)$ are contained in a disk of radius
$\rho(d(G))$, where $d(G)$ is the maximal degree of the vertices of $G$.
This is true for the characteristic polynomial and its Laplacian version, \cite[Chapter 3]{bk:BrouwerHaemers2012}.
This is also 
the case for the chromatic polynomial,
\cite{bk:DongKohTeo2005,ar:Sokal01}, but the proof of this is far from trivial.
\item
For every $G$ the roots of $P(G;X)$ are contained in a disk of constant radius.
This is the case for the edge-cover polynomial, \cite{ar:CsikvariOboudi2011}.
For the unit disk this is the univariate version of Schur-stability.
\item
The roots of $P(G;X)$ are dense in the complex plane.
This is again true for the chromatic polynomial, the dominating polynomial and the independence
polynomial, \cite{bk:DongKohTeo2005,ar:Sokal04,ar:BrownHickmanNowakowski2004,phd:Hoshino}.
\end{enumerate}

In \cite{ar:MakowskyRavveBlanchard2014} we showed that the precise location of roots of univariate
$\SOL$-definable
graph polynomials is not a graph theoretic (semantic) property of graphs.
In the next subsection
we investigate whether stability, the multivariate analog the location of zeros,
of multivariate 
$\SOL$-definable graph polynomials is a semantic property.
A typical theorem from \cite{ar:MakowskyRavveBlanchard2014} is the following.
\begin{theorem}[{\cite[Theorem 4.22]{ar:MakowskyRavveBlanchard2014}}]
\label{th:main3}
For every univariate graph polynomial $P(G;X) \in \SOLEVAL$ there exists a univariate graph polynomial $Q(G;X)$
which is prefactor equivalent to $P(G;X)$ and the roots of $Q(G;X)$ are dense in $\C$.
Furthermore, if $P(G;X) \in \SOLEVAL$ so is $Q(G;X)$.
\end{theorem}
To show this we use \cite[Lemma 4.21]{ar:MakowskyRavveBlanchard2014}:
\begin{lemma}
\label{le:dense-C}
There exist an
$\SOL$-definable 
univariate similarity polynomial 
$D_{\C}(G;X)$ of degree $48$ such that all its roots are dense in $\C$.
\end{lemma}

\begin{proof}[Proof of Theorem \ref{th:main3}:]
We use Lemma \ref{le:dense-C} and put
$$
Q(G;X) = D(G;X) \cdot P(G;X)
$$
and the fact that $\SOLEVAL$ is closed under products.
\end{proof}

\subsection{Stable multivariate graph polynomials}
\label{se:stable}
\label{se:intro-stable}
A multivariate polynomial is {\em stable}\footnote{
Multivariate analogs of location of zeros of polynomials are the various
halfplane properties aka stability properties.

In engineering and stability theory, a square matrix $A$ is called stable matrix 
(or sometimes Hurwitz matrix) if every eigenvalue of $A$ has strictly negative real part. 
These matrices were first studied in the landmark paper \cite{ar:Hurwitz1895} in 1895.
The Hurwitz stability matrix plays a crucial part in control theory. 
A system is stable if its control matrix is a Hurwitz matrix. 
The negative real components of the eigenvalues of the matrix represent negative feedback. 
Similarly, a system is inherently unstable if any of the eigenvalues have positive real components, 
representing positive feedback.
In the engineering literature, one also considers Schur-stable univariate polynomials, which are polynomials
such that all their roots are in the open unit disk, see for example \cite{ar:wang-schur-stability1994}.
} 
if the imaginary part of its zeros is negative,
and it is {\em Hurwitz-stable} if the real part of its zeros is negative.
Analogously, it is {\em Schur-stable} if all its roots are in the open unit ball.
Recently, stable and Hurwitz-stable polynomials have attracted the attention of
combinatorial research. In \cite{ar:COSW-2004} the study of graph and matroid invariants
and their various stability properties was initiated. 
The more recent paper \cite{ar:HirasawaMurasugi2013} 
does the same for knot and link invariants.
Due mainly to the recent work of 
J. Borcea and P. Br\"and\'en \cite{ar:BorceaBraenden2008}, see also \cite{ar:Wagner2011}, a 
very successful multivariate generalization
of stability of polynomials has been developed.
To quote from the abstract of \cite{pr:Vishnoi-2013}:
\begin{quote}
Problems in many different areas of mathematics reduce to questions about the zeros of
complex univariate and multivariate polynomials. Recently, several significant and seemingly
unrelated results relevant to theoretical computer science have benefited from taking this route:
they rely on showing, at some level, that a certain univariate or multivariate polynomial has
no zeros in a region. This is achieved by inductively constructing the relevant polynomial
via a sequence of operations which preserve the property of not having roots in the required
region.
\end{quote}
Further on, \cite{pr:Vishnoi-2013} gives the following applications  of stable polynomials to theoretical computer science:
A new proof of the van der Waerden conjecture about the permanent of doubly stochastic matrices, \cite{pr:Gurvits2006};
various applications to the traveling salesman problem, 
\cite{pr:Vishnoi2012}, \cite{ar:Pemantle2012}; 
applications to the Lee-Yang theorem 
in statistical physics that shows the lack of phase
transition in the Ising model, 
\cite{pr:SinclairSrivastava2013},
 and more.
\cite{borcea2009negative} discuss various sampling problems and show, among other things, 
that the generating polynomial of spanning trees of a graph
is stable, see also \cite{anari2016monte}.
Let $m,n \in \N$ be indices.
Let 
$\mathbf{X} =( X_1, \ldots , X_n)$ 
and
$\mathbf{Y} =( Y_1, \ldots , Y_m)$ 
be $n+m$ indeterminates and $f(\bX,\bY) \in \C[\bX, \bY]$.
Let $\cH_u = \{ a \in \C: \Im(a) > 0\}$
and $\cH_r = \{ a \in \C: \Re(a) > 0\}$ be the upper, respectively right half-plane of $\C$.

\begin{Definitions}
\begin{enumerate}
\item
$f$ is homogeneous if all its monomials have the same degree.
\item
$f$ is multiaffine if each indeterminate occurs at most to the first
power in $f$.
\item
$f(\bX,\bY) \in \C[\bX, \bY]$ is {\em stable} if $f \equiv 0$ or, whenever $\mathbf{a} \in \cH_u^{n+m}$, then
$f(\ba) \neq 0$. If additionally $f(\bX, \bY) \in \R[\bX,\bY]$, it is real stable.
\item
$f$ is {\em  Hurwitz-stable} if $f \equiv 0$ or, whenever $\mathbf{a} \in \cH_r^{n+m}$, then
$f(\ba) \neq 0$.
\item
$f$ is {\em stable with respect to $\bX$} if for every $\bb \in \cH^{m}$ either 
$f(\bX, \bb) \equiv 0$ or whenever $\ba \in \cH_u^{n}$ then
$f(\ba, \bb) \neq 0$.
\item
Let $\cK$ be class of finite graphs.
A graph polynomial $P(G; \mathbf{X})$ is {\em stable on $\cK$} if for every
graph $G \in \cK$ the polynomial $P(G;\bX) \in \C[\bX]$ is stable.
\end{enumerate}
\end{Definitions}

\begin{remark}
If $f(\bX,\bY)$ is stable, it is stable with respect to $\bX$, but not conversely. 
\end{remark}

\begin{Examples}
\label{ex:stable}
\begin{enumerate}
\item
Univariate polynomials are stable iff they have only real roots.
\item
The characteristic polynomial $P_{A}$ and its Laplacian version $P_L$ 
are stable because they have only real roots.
\item
Let $\mathrm{Tree}(G;X) = \sum_{T \subseteq E(G)} \prod_{e \in T} X_e$,
be the tree polynomial, 
where $T$ ranges over all trees of $G=(V(G),E(G))$.
$\mathrm{Tree}(G;X)$ is Hurwitz-stable, \cite{ar:COSW-2004}.
\item
\label{ex:sokalization}
Let $G=(V(G),E(G))$ be a graph and let $\bX_E= (X_e: e \in E(G))$ be commutative indeterminates.
Let $S$ be a family of subsets of $E(G)$, i.e., $S \subset \wp(E(G))$ and 
let $\mathrm{P}_S(G;\bX_E) = \sum_{A \in S} \prod_{e \in A} X_e$. 
If $S$ is the family of trees of $E(G)$
then $\mathrm{P}_S(G;\bX_E)$
is  a multivariate
version of the tree polynomial,
which is also Hurwitz-stable, cf. \cite[Theorem 6.2]{ar:Sokal2005a}.
\item
In \cite[Question 1.3]{ar:COSW-2004} it is asked for which $S$ is the polynomial $\mathrm{P}_S(G;\bX_E)$
Hurwitz-stable. Actually they ask the corresponding question for matroids 
$$M=(E(M),S(M)).$$
\item
In \cite[Section 16]{ar:HirasawaMurasugi2013} the stability of multivariate knot polynomials is studied.
\end{enumerate}
\end{Examples}

\subsection{Sufficient conditions for stability}
The characteristic polynomial of a symmetric real matrix is stable. Stable polynomials are often determinant like
in the following sense:

\begin{theorem}[Criteria for Stability]
\label{th:determinants}
Let $\bX =(X_1, \ldots , X_m)$ be indeterminates,
and $\mathcal{X}$ be the diagonal matrix of $n$ indeterminates with $(\mathcal{X})_{i,i} = X_i$.
\begin{enumerate}
\item \label{th:det-1}
(\cite[Proposition 2.4]{ar:BorceaBraenden2008})
For $i \in [m]$ let each $A_i$ be a positive semi-definite Hermitian $(n \times n)$-matrix
and let $B$ be Hermitian. Then
$$
f(\bX) = \det(X_1 A_1 + \ldots + X_m A_m +B) \in \R[\bX]
$$
is stable.
\item \label{th:det-2}
(\cite[Theorem 2.2]{ar:HeltonVinnikov2007})
For $m=2$ and $f(X_1,X_2) \in \R[X_1,X_2]$ we have 
$f(X_1, X_2)$ is stable iff there are Hermitian matrices $A_1, A_2, B$
with $A_1, A_2$ positive semi-definite such that
$$
f(X_1,X_2) = \det(X_1 A_1 +  X_2 A_2 +B). 
$$
\item \label{th:det-3}
(\cite[after Theorem 4.2]{ar:Braenden2007})
If $A$ is a Hermitian $(m \times m)$ matrix then the polynomials $\det(\mathcal{X} + A)$ and
$\det(\mathbf{1} + A \cdot \mathcal{X})$ are real stable.
\end{enumerate}
\end{theorem}
\begin{theorem}[Criteria for Hurwitz-stability]
\label{th:Hurwitz}
\begin{enumerate}
\item
\label{th:homo}
(\cite{ar:WagnerWei2009})
If $f(\bX) \in \R[\bX]$ is  a real homogeneous  then $f(\bX)$ is stable iff  $f(\bX)$ is Hurwitz stable.
\item \label{th:det-4}
(\cite[Theorem 8.1]{ar:COSW-2004})
Let $A$ be a  complex $(r \times m)$-matrix, $A^*$ be its Hermitian conjugate, 
then  the polynomial in $m$-indeterminates
$$
Q(\bX)  = \det(A\mathcal{X}A^*)
$$
is multiaffine, homogeneous and Hurwitz-stable.
\item \label{th:det-5}
(\cite[after Theorem 4.2]{ar:Braenden2007})
If $B$ is a skew-Hermitian $(n \times n)$ matrix then $\det(\mathcal{X} + B)$ and $\det(\mathbf{1} + B \cdot \mathcal{X})$
are Hurwitz stable.
\item \label{th:per}
(\cite[Theorem 10.2]{ar:COSW-2004})
Let $A$ be a  real $(r \times m)$-matrix with non-negative entries.
Then  the polynomial in $m$-indeterminates
$$
Q(\bX)  = \per(AX) =
\sum_{S \subseteq [m], |S| =r} \per( A\mid_S) \prod_{i \in S} X_i
$$
is Hurwitz-stable.
\end{enumerate}
\end{theorem}

\subsection{Making graph polynomials stable}
We first consider graph polynomials with a fixed number of indeterminates $m$.
Let $P(G;\bX)$ be a graph polynomial with integer coefficients and with $\SOL$-definition
$$
P(G;\bX) = \sum_{\phi} \prod_{\psi_1} X_1 \cdot \ldots \cdot \prod_{\psi_m} X_m,
$$
with coefficients $(c_{i_1, \ldots , i_m}: i_j \leq d(G), j \in [m])$
$$
P(G;\bX) = \sum_{i_1, \ldots , i_m} c_{i_1, \ldots , i_m} X_1^{i_1}X_2^{i_2} \ldots X_m^{i_m} \in \N[\bX],
$$
such that in each indeterminate the degree of $P(G, \bX)$ is less than $d(G)$.
We put $M(G) = d(G)^m$ which serves as a bound on the number of relevant coefficients,
some of which can be $0$.

\begin{theorem}
\label{th:main-stable}
There is a stable graph polynomial $Q^s(G; Y, \bX)$ with integer coefficients such that
\begin{enumerate}
\item
the coefficients of
$Q^s(G)$ can be computed uniformly\footnote{
There is a polynomial time computable function
$F: \Z[\bX] \rightarrow \Z[Y, \bX]$ such that for all graphs $G$ we have
$F(P(G;\bX)) = Q^s(G;Y,\bX)$.
} 
from  the coefficients of $P(G)$
in time polynomial in the size of the encoding of the coefficients;
\item
there is $a_0 \in \N$ such that
$Q^s(G; a_0, \bX)$ is d.p.-equivalent to $P(G;\bX)$;
\item
$Q^s(G; Y, \bX)$ is $\SOL$-definable and its $\SOL$-definition
can be computed uniformly 
from $\phi, \psi_1, \ldots, \psi_m$
in time polynomial in the size of the formulas $\phi, \psi_1, \ldots, \psi_m$.
\end{enumerate}
\end{theorem}

\begin{theorem}
\label{th:main-hu-stable}
If additionally, $P(G;\bX)$ has only non-negative coefficients,
there is a Hurwitz-stable graph polynomial $Q^h(G; Y, \bX)$ with non-negative integer coefficients  and one more
indeterminate $Y$ such that
\begin{enumerate}
\item
the coefficients of $Q^h(G)$ can be computed uniformly in polynomial time from  the coefficients of $P(G)$;
\item
there is $\ba \in \N^{M(G)-n}$ such that
$Q^h(G; \ba, \bX)$ is d.p.-equivalent to $P(G;\bX)$;
\item
$Q^h(G; Y, \bX)$ is $\SOL$-definable and its $\SOL$-definition
can be computed uniformly in polynomial time from $\phi, \psi_1, \ldots, \psi_m$.
\end{enumerate}
\end{theorem}

In \cite{ar:COSW-2004,ar:Sokal2005a} the authors also consider graph polynomials
where the number of indeterminates depends on the graph $G=(V(G),E(G))$, as in
Example \ref{ex:stable}(\ref{ex:sokalization}).
We will not give the most general definition here, but restrict ourselves to
the case the indeterminates $X_e$ are labeled by the edges $E(G)$ of $G$.
We put $m(G)$ to be the cardinality of $E(G)$.

Let $S(G;\bX)$ be a multiaffine graph polynomial with non-negative integer coefficients and with $\SOL$-definition
$$
S(G;\bX) = \sum_{\phi(A)} \prod_{\psi_1(A,e)} X_e \cdot \ldots \cdot \prod_{\psi_m(A,e)} X_e,
$$
and coefficients 
$(c_{i_1, \ldots , i_m}: i_j \in \{0,1\}, j \in [m(G)])$
$$
S(G;\bX) = \sum_{i_1, \ldots , i_m} c_{i_1, \ldots , i_m} X_1^{i_1}X_2^{i_2} \ldots X_m^{i_{m(G)}} \in \N[\bX],
$$
such that in each indeterminate the degree of $S(G, \bX)$ is less than $d(G)$.
We put $M(G) = 2^{m(G)}$ which serves as a bound on the number of relevant coefficients,
some of which can be $0$.
Let $\bX_G =(X_e: e \in E(G))$.

\begin{theorem}
\label{th:sokalization}
There are graph polynomials 
$T^s(G; \bX_G)$  and
$T^h(G; \bX_G)$ 
with non-negative integer coefficients such that
\begin{enumerate}
\item
$T^s(G; \bX_G)$ is stable and
$T^h(G; \bX_G)$ is Hurwitz-stable; 
\item
Both the coefficients of  $T^s(G; \bX_G)$  and of  $T^h(G; \bX_G)$ 
can be computed uniformly in polynomial time from the coefficients of $S(G; \bX_G)$;
\item
Both
$T^s(G; \bX_G)$  and $T^h(G; \bX_G)$ 
are d.p.-equivalent to $S(G;\bX_G)$;
\item
Both
$T^s(G; \bX_G)$  and $T^h(G; \bX_G)$ 
are $\SOL$-definable and its $\SOL$-definition
can be computed uniformly in polynomial time from $\phi, \psi_1, \ldots, \psi_m$.
\end{enumerate}
\end{theorem}
\subsection{Proofs}
The proofs have two components: One uses first some {\em dirty trick}
to modify the polynomial such that it behaves as required, but then one has to show that
this {\em dirty trick} can be performed in way which preserves definability in $\SOLEVAL$.

\subsubsection{Proof of Theorem \ref{th:main-stable}}

We use Theorem \ref{th:determinants}(\ref{th:det-1}). 
Let $\alpha: \N^m \rightarrow \N$  which maps $(i_1, \ldots i_m) \in \N^m$
into its position in the lexicographic order of $\N^m$. 
We relabel the coefficients of $P(G;\bX)$
such that $d_i = c_{i_1, \ldots , i_m}$ with $\alpha(i_1, \ldots , i_m) =i, i \in [M]$ and $M=d(G)^m$.

We put $B$ to be the $(M \times M)$ diagonal matrix with $B_{i,i} = d_i \cdot Y_i$ and $A_1= A_2= \ldots = A_m$
to be the $(M \times M)$ identity matrix. 
The identity matrix is both Hermitian and positive semi-definite. Furthermore, $B \mid_{Y=a}=B(a)$ being a diagonal matrix,
is Hermitian for every $a \in \C$. 
Hence, 
$$
Q_a^s(G;a, \bX) = \det(B(a) + \sum_{i=1}^M X_i \cdot A_i) = \prod_{i=1}^M (d_i + \sum_{i=1}^M X_i)
$$
is stable for every $a \in \C$.

We have to verify (i)-(iii).

(i):
All the matrices can be computed in
polynomial time in  $\Z[Y, \bX]$.

(ii):
We use Theorem \ref{th:equiv}:
$Q^s \preceq_{d.p} P$ follows from (i).
We have to show that there is $a_0\in\N$ with $P \preceq_{d.p} Q_{a_0}^s$.
The function $\alpha$ can be easily inverted.
To recover  the coefficients of $P(G)$ from the coefficients of $Q^s(G)$, we note that 
$$
\sum_{i=0}^M d_i(G) \cdot Y^i
$$ 
is the coefficient of 
$(\sum_{\ell=1}^{m} X_{\ell})^{M-1}$ of $Q^s(G; Y, \bX)$.
This can be computed in polynomial time from  the coefficients of $Q^s$.
To find $a_0$ we let $a_0 \in \N$ be bigger than $1+2 \cdot |d_i(G)|$, as $d_i(G)$ could be negative.
Now $\sum_{i=0}^M d_i(G) \cdot a_0^i$ can be viewed as a natural number written in base  $a_0$,
and the digits  $d_i(G)$ can be uniquely determined.

(iii): To prove that  $Q^s(G;Y, \bX)$ is $\SOL$-definable we need a few lemmas from 
\cite{ar:FischerKotekMakowsky11,ar:KotekMakowskyZilber11,phd:Kotek}.

The first lemma is part of the definition of $\SOL$-definability.
\begin{lemma}
\label{le:sol-1}
Finite sums and products of $\SOL$-definable polynomials are $\SOL$-definable.
\end{lemma}

\begin{lemma}
\label{le:sol-2}
Let $G_< = (V(G), E(G), <(G))$ be a graph with an ordering $<(G)$ on the vertices.
Let $Q(G;\bX)$ be a graph polynomial with non-negative integer coefficients and with $\SOL$-definition
$$
Q(G;\bX) = \sum_{A \subseteq V^r:\phi(A)} 
\prod_{\mathbf{v}_1 \in A:\psi_1(A, \mathbf{v}_1)} X_1 \cdot \ldots \cdot 
\prod_{\mathbf{v}_m \in A:\psi_1(A, \mathbf{v}_m)} X_m 
$$
with coefficients  $(c_{i_1, \ldots , i_m}: i_j \leq d(G), j \in [m])$
$$
Q(G;\bX) = \sum_{i_1, \ldots , i_m} c_{i_1, \ldots , i_m} X_1^{i_1}X_2^{i_2} \ldots X_m^{i_m} \in \N[\bX],
$$
such that in each indeterminate the degree of $P(G, \bX)$ is less than $d(G)$.

Let $s(G)$ be such that $|V(G)|^{s(G)} \geq d(G)$ and extend the ordering $<(G)$ to the lexicographic
ordering of $|V(G)|^{s(G)}$. For $\mathbf{v} \in V(G)^{s(G)}$ we define $Init(G;\mathbf{v})$ to be the set
of predecessors of $\mathbf{v}$ in this lexicographic ordering.

The coefficients $c_{i_1, \ldots , i_m}$ of $Q(G;\bX)$ are $\SOL$-definable by
$$
c(\mathbf{v}_1, \ldots , \mathbf{v}_m) = \sum_{A \subseteq V^r} 1,
$$
where $A$ ranges over all subsets satisfying $\phi(A)$ and for each $\ell \in [m]$
the set $Init(G;\mathbf{v}_{\ell})$ is of the same size as $i_{\ell}$ and as
$$
\{ 
\mathbf{w}_{\ell} \in V^r: 
(V(G),E(G),<(G),A, \mathbf{w}_{\ell}) 
\models 
\phi(A) \wedge \psi(A, \mathbf{w}_{\ell}) 
\}. 
$$
\end{lemma}

\begin{proof}
We only have to note that the equicardinality requirement is expressible in $\SOL$.
\end{proof}

\begin{lemma}
\label{le:sol-3}
The polynomial
\begin{gather}
Q_a^s(G;a, \bX) = 
\prod_{i=1}^M (d_i + \sum_{i=1}^M X_i) =
\prod_{\mathbf{v}_1, \ldots , \mathbf{v}_m} \left(c(\mathbf{v}_1, \ldots , \mathbf{v}_m) +\sum_{i=1}^M X_i \right)
\notag
\end{gather}
is $\SOL$-definable.
\end{lemma}

\subsubsection{Proof of Theorem \ref{th:main-hu-stable}}
Now all the coefficients of $P(G;\bX)$ are non-negative.
We want to use Theorem \ref{th:Hurwitz}(\ref{th:homo}) together with 
Theorem \ref{th:determinants}(\ref{th:det-1}). 
We repeat the proof of
Theorem \ref{th:main-stable} with the following changes:
Let $D$ be the diagonal $(M\times M)$-matrix of the coefficients, and $Y$ a new indeterminate.
Instead of $B(a)$ we use $D\cdot Y$ where $Y$ is now a scalar.
$D$ is now a diagonal matrix with non-negative coefficients, so it is positive semi-definite.
We put 
$$
Q(G; Y,\bX)= 
 \det(D\cdot Y + \sum_{i \in [m]} A_i \cdot X_i).
$$ 
The resulting polynomial $Q(G; Y, \bX)$ is homogeneous and has integer coefficients.
So we can apply 
Theorem \ref{th:Hurwitz}(\ref{th:homo}) together with 
Theorem \ref{th:determinants}(\ref{th:det-1}) to make to see that 
$Q(G; Y,\bX)$ is both stable and Hurwitz stable.
In particular, for each $a \in \N$
$Q(G; a,\bX)$ is Hurwitz stable.
To see that
$Q(G;Y,\bX)$ is $\SOL$-definable we again use $a \in \N$ large enough as in the proof of
Theorem \ref{th:main-hu-stable}.

\subsubsection{Proof of Theorem \ref{th:sokalization}}
The proof is the same as the proof of 
Theorem \ref{th:main-hu-stable}, where the number of indeterminates equals the number $m(G) = \mid E(G)\mid$.

\section{Conclusions and open problems}
\label{se:conclu}

In this paper we presented the logician's view of graph polynomials. This includes
model theoretic reinterpretations of some of our previous work on graph polynomials,
such as
\cite{ar:KotekMakowskyZilber11,ar:FischerKotekMakowsky11,phd:Kotek,ar:MakowskyRavve2013,ar:MakowskyRavveBlanchard2014,ar:KotekMakowskyRavve2017arxiv}.
We systematically studied  various notions of semantic equivalence of graph polynomials
based on the notion of distinctive power. We were careful to set up this logical framework
to be consistent with the way graph polynomials are compared in the graph theoretic literature.
We also discussed various forms of graph polynomials, and unified all these under the
framework of $\SOL$-definable graph polynomials.
Within this framework we also have a Normal Form Theorem \ref{th:normal-form}.

In \cite{ar:MakowskyRavve2013,ar:MakowskyRavveBlanchard2014} we initiated
the study of semantic equivalence of univariate graph polynomials without focusing
on definability or complexity. We showed there that the location of the roots
are not a semantic property.

In this paper we have extended these studies  to multivariate graph polynomials.
We have also extended our framework threefold:
\begin{enumerate}
\item
We have imposed computability 
restriction on our framework.
To have a workable framework
it does not suffice that the coefficients of a graph polynomial
have to be computable from the graph, but that one needs to require that
the inverse problem be decidable as well.
This additional requirement was not used in
\cite{pr:MakowskyKotekRavve2013}, where we were more concerned with complexity issues
of evaluating graph polynomials.
\item
We have restricted our discussion to $\SOL$-definable graph polynomials.
This means that the d.p.-equivalent polynomial with stability properties
has to be $\SOL$-definable as well.
In the univariate cases discussed in
\cite{ar:MakowskyRavve2013,ar:MakowskyRavveBlanchard2014} the additional definability
requirement is not too difficult to be established.
In the multivariate case, this is considerably more complicated.
\item
We have studied stability and Hurwitz-stability (aka the half-plane property) 
of multivariate graph polynomials.
We have chosen this topic, because various graph polynomials arising from modeling
natural phenomena turn out to be stable or Hurwitz-stable.
Our study shows that these stability properties do not really reflect properties of
the underlying graphs proper, but are the result of extraneous requirements  arising from
the particular modeling process of the natural phenomena in question.
\end{enumerate}

Our work shows that to justify the study of the location of the zeroes of
a graph polynomial, the particular choice of the coefficients of the graph polynomial
has to be taken into account. If the only purpose of the graph polynomial is to encode
purely graph theoretic properties, the location of its zeroes is irrelevant.

\ifskip\else
\subsection{Complexity issues}
\label{se:compl-discuss}
We have not gone into details concerning complexity.
\begin{definition}
Let $P$ and $Q$ be two d.p.-equivalent graph polynomials.
We say that $P$  is
{\em polynomially d.p.-reducible (P.d.p-reducible)}
to $Q$ if $F$ from Theorem \ref{prop:equiv} is computable in polynomial time.
$P$ and $Q$  P.d.p.-equivalent if they are P.d.p.-reducible to each other.
\end{definition}
All our theorems  and examples changing the locations of zeroes produce P.d.p.-equivalent polynomials.

We have avoided giving a definition of the complexity of a computable or $\SOL$-definable $P(G;\bX)$
graph polynomial because of our requirement in Definition \ref{def:comp} that the range of $P$ be computable.
In order to get a reasonable framework one would have to require, that for $P$ and $Q$ both computable
in some time complexity class 
$\mathrm{C}$
and $P$ d.p.-reducible to $Q$, the function $F$ from Theorem \ref{prop:equiv}
should also be computable in  time in
$\mathrm{C}$.
To achieve this we would have to require that given $s \in \Z[\bX]$ finding a graph $G_s$ such that
$P(G_s,\bX)=s$ can be done also in  time in
$\mathrm{C}$.
This is a stronger requirement than
limiting the complexity of
$\beta_P$ from Proposition \ref{prop:beta}. 
Furthermore, it is unlikely to be true for reasonably interesting graph polynomials.

\subsection{Open problems}
The landmark paper
\cite{ar:JaegerVertiganWelsh90} initiated the study of the complexity of evaluating
the Tutte polynomial, cf. also \cite{bk:Welsh93}.
In \cite{pr:MakowskyKotekRavve2013} we survey later developments of this approach.
Surprisingly, the complexity of the inverse problem was, to the best of our knowledge, not studied
even for the most prominent graph polynomials.

Let $P(G;\bX)$ be a computable $\SOL$-definable graph polynomial.
The {\em Recognition Problem for $P(G;\bX)$} is the question
\begin{quote}
Given a polynomial $s(\bX) \in \Z[\bX]$
is there a graph $G_s$ such that $P(G_s;\bX)=s(\bX)$?
\end{quote}

\begin{problem}
\label{inverse-problem}
\begin{enumerate}
\item
What is the complexity of the  Recognition Problem for $P(G;\bX)$ for
$\SOL$-definable graph polynomials?
\item
For all $\SOL$-definable graph polynomials studied in the literature, the Recognition Problem can be solved
in exponential time.
Is this true for all $\SOL$-definable graph polynomials?
\item
What is the exact complexity  of the Recognition Problem for specific graph polynomials, such as the
the polynomials $P_{A}$ and $P_L$ from spectral graph theory 
given in Example \ref{ex:prominent}(\ref{ex:spectral}),
chromatic polynomial, or the Tutte polynomial?
\end{enumerate}
\end{problem}
It is easy to define graph polynomials $P(G;X)$ with a {\em trivial Recognition Problem}, i.e. where for every
polynomial 
$$s(X) = \sum_{i=0}^m a_i X^i \in \N[X]$$ 
there is a graph $G_s$ with $P(G_s;X)= s(X)$.

Let 
$$MaxCl(G;X) = \sum_i mcl_i(G) X_i$$
be the  graph polynomial 
where $mcl_i(G)$ denotes the number of maximal cliques of size $i$.
Let $s(X) = \sum_{i=0}^m a_i X^i \in \N[X]$ and let $G_s$ be the graph which is the disjoint union
of $a_i$-many cliques of size $i$. Then $MaxCl(G_s;X) = s(X)$.

\begin{problem}
Characterize the $\SOL$-definable graph polynomials with trivial Recognition Problem.
\end{problem}

\begin{problem}
If $P(G;\bX)$ is $\SOL$-definable and its Recognition Problem is decidable 
in polynomial time, what is the complexity of finding, given $s(X) = \sum_{i=0}^m a_i X^i \in \N[X]$,
a graph $G_s$ with $P(G_s;X)= s(X)$?
\end{problem}

\fi 

\subsection*{Acknowledgment}
The authors would like to thank Petter Br\"and\'en for
guiding us to the literature of stable polynomials, and Jason Brown
and four anonymous readers of an earlier version of this paper for valuable comments.
Thanks also to Jingcheng Lin for pointing out the references \cite{borcea2009negative} and \cite{anari2016monte}.
We want to acknowledge that some of the definitions and examples
were taken verbatim from 
\cite{ar:KotekMakowskyZilber11,phd:Kotek,ar:MakowskyRavveBlanchard2014} and \cite{ar:KotekMakowskyRavve2017}.
D.p-equivalence was first characterized in \cite{up:Lecture-11}.
\section*{References}
\newcommand{\etalchar}[1]{$^{#1}$}

\end{document}